\documentclass[11pt]{article} %
\usepackage{fullpage}
\usepackage{amsthm}
\usepackage{amsmath, amssymb, natbib, graphicx, url}
\usepackage{color}
\usepackage{todonotes}
\usepackage{dsfont}
\usepackage{caption}
\usepackage{subcaption}
\usepackage{comment}
\usepackage{hyperref}
\definecolor{refcol}{rgb}{0.1,0,0.6}
\hypersetup{colorlinks}
\hypersetup{linkcolor=refcol, citecolor=refcol}
\usepackage{wrapfig}
\usepackage{pifont}
\usepackage{algorithm}
\usepackage{algpseudocode}
\usepackage[nameinlink]{cleveref}

\newcommand{\cA}{\mathcal{A}}

\newcommand{\cF}{\mathcal{F}}

\newcommand{\cJ}{\mathcal{J}}

\newcommand{\cO}{\mathcal{O}}

\newcommand{\N}{\mathbb{N}}

\newcommand{\R}{\mathbb{R}}

\usepackage{enumitem}
\usepackage[normalem]{ulem}

\newcommand{\dd}{{\rm d}}

\usepackage{mathtools}

\usepackage{authblk}

\newtheorem{theorem}{Theorem}[section]

\newtheorem{definition}[theorem]{Definition}

\newtheorem{lemma}[theorem]{Lemma}

\newtheorem{proposition}[theorem]{Proposition}
\newtheorem{remark}[theorem]{Remark}

\newtheorem{assumption}[theorem]{Assumption}

\title{Ensemble Kalman inversion with non-smooth regularization}
\author[]{Simon Weissmann}
\affil[]{\normalsize
  Universit\"at Mannheim, Institute of Mathematics\\

  68138 Mannheim, Germany\\
  
\texttt{simon.weissmann@uni-mannheim.de}
}
\date{}

\begin{document}

\maketitle

\begin{abstract}
This paper investigates ensemble Kalman inversion (EKI) for variational inverse problems with convex, potentially non-smooth regularization. 
While deterministic EKI and its Tikhonov-regularized variants have primarily been analyzed for smooth objectives, a corresponding framework accommodating subgradient dynamics has not yet been established. 
To address this gap, we introduce a subgradient-based formulation of EKI (SEKI) that incorporates non-smooth regularizers through a covariance-preconditioned differential inclusion for the ensemble mean.
In the linear forward-model setting, well-posedness of the resulting continuous-time particle system is established under minimal assumptions on the regularization functional using maximal monotone operator theory and Yosida approximations. 
Motivated by the continuous-time dynamics, we propose an explicit discrete-time scheme that preserves the derivative-free structure of EKI and analyze its convergence as an optimization method in the strongly convex case.
Numerical experiments in computed tomography with total variation regularization and sparse recovery with $\ell_1$ penalties illustrate that non-smooth regularization can be incorporated into ensemble Kalman inversion in a stable and principled manner.
\end{abstract}

\section{Introduction}
Inverse problems arise in a wide range of applications in science and engineering, including medical imaging, geophysics, and parameter estimation in partial differential equations. The goal is to recover an unknown quantity of interest from indirect and noisy observations. 
Due to ill-posedness, inverse problems typically require regularization in order to ensure stability with respect to perturbations in the data. 
A common approach is to formulate inverse problems as variational optimization problems of the form
\[
\min_{x \in \mathcal X} \; \mathcal L(G(x),y) +R(x),
\]
where $G:\mathcal X\to \R^{K}$ denotes the forward operator mapping from the parameter space $\mathcal X$ to the observation space $\R^K$, $\mathcal L$ is a data misfit functional, $y\in\R^K$ denotes the observed data, and $R$ is a regularization term encoding prior information. 
While classical Tikhonov regularization leads to smooth optimization problems, many practically relevant regularizers, such as total variation or sparsity-promoting $\ell^1$ penalties, are inherently non-smooth. 
This motivates optimization methods that can accommodate non-smooth objectives.

An alternative perspective is provided by the Bayesian formulation of inverse problems, in which the unknown parameter is modeled as a random variable with a prescribed prior distribution. 
Conditioning on the observed data yields the posterior distribution, which typically cannot be computed analytically and must be approximated numerically. 
Ensemble-based methods have emerged as an attractive class of algorithms in this context.

The ensemble Kalman filter (EnKF), originally introduced for data assimilation, has been adapted to inverse problems under the name ensemble Kalman inversion (EKI). EKI propagates an ensemble of particles through a sequence of updates based on empirical means and covariances, avoiding the explicit computation of derivatives of the forward map. 
This derivative-free structure has made EKI particularly attractive for large-scale and nonlinear inverse problems. 
Although EKI is often motivated from a Bayesian viewpoint, it is now well understood that it does not, in general, sample from the posterior distribution in nonlinear settings \cite{ErnstEtAl2015}. 
Instead, deterministic formulations of EKI are more naturally interpreted as optimization methods once stochastic perturbations of the data are removed \cite{CST2019,SS17}. 
In the continuous-time limit, the deterministic EKI can be written as a system of coupled ordinary differential equations for the ensemble members. 
For linear forward models with quadratic regularization, these dynamics admit an exact interpretation as a preconditioned gradient flow, where the empirical ensemble covariance acts as an adaptive preconditioner \cite{SS17}. 
For nonlinear forward models, this gradient-flow structure holds only approximately. 
This observation has motivated a line of work interpreting EKI as a derivative-free gradient-like optimization method \cite{Chada2019ConvergenceAO,tong2022localization,YL21,SRT2025,Weissmann_2022}.\medskip

The existing gradient-flow analysis of EKI relies crucially on smoothness of the objective. 
However, many practically relevant regularization strategies, such as $\ell^1$ penalties or total variation, are non-smooth, and gradients are no longer well defined. In these situations the classical gradient-flow interpretation of EKI breaks down. 
This raises the following question:
\begin{center}
\textit{How can non-smooth convex regularization be incorporated into ensemble Kalman inversion\\ in a principled and well-posed manner?}
\end{center}
In this work, we address this question by extending the gradient-flow framework of EKI to the non-smooth setting. Replacing classical gradients by subgradients leads naturally to a formulation in terms of differential inclusions. This results in a subgradient-based ensemble Kalman inversion (SEKI), which can be interpreted as a covariance-preconditioned subgradient method whose preconditioner is learned adaptively from the ensemble. For this formulation we establish well-posedness of the continuous-time dynamics, analyze the resulting discrete-time algorithm, and study the role of ensemble collapse.

\subsection{Related work}

The EnKF was introduced by Evensen \cite{GE03} 
as a particle-based approximation of the filtering distribution arising in data assimilation problems. 
It was subsequently applied to Bayesian inverse problems 
\cite{chen2012ensemble,emerick2013ensemble}, 
where it is commonly referred to as EKI.

A substantial body of work studies the large-ensemble limit of ensemble Kalman methods. 
For linear Gaussian models, convergence and consistency results have been established in \cite{L2009LargeSA,doi:10.1137/140965363}. 
Extensions to nonlinear models have been investigated in \cite{doi:10.1137/140984415}, and ensemble square-root formulations have been analyzed in \cite{LS2021_c}. 
Mean-field limits and associated Fokker–Planck descriptions have been studied in \cite{DL2021_b,MHGV2018}. 
In particular, for linear forward models with Gaussian priors, EKI approximates the posterior distribution in the mean-field limit. 
Error analysis for finite ensemble sizes has been carried out for the EnKF in \cite{Tong2018,doi:10.1002/cpa.21722} and for the ensemble Kalman–Bucy filter in \cite{delmoral2018,doi:10.1137/17M1119056}. 
Viewing the ensemble as a Monte Carlo approximation has also motivated the development of multilevel variants \cite{doi:10.1137/15M100955X,HST2020}.

Another important research direction concerns continuous-time formulations of ensemble Kalman methods. 
Continuous-time limits were formally derived in \cite{SS17} and subsequently analyzed rigorously in \cite{BSW18,BSWW2021}. 
Related formulations include localized and mollified variants \cite{bergemann2010localization,bergemann2010mollified,Reich2011,LS2021,LS2021_b,lange2021derivation}. 
In deterministic settings obtained by neglecting stochastic perturbations of the data, well-posedness and convergence results have been established for linear models in \cite{SS17} and for stochastic formulations in \cite{BSWW19}. 
A spectral characterization of the deterministic EKI dynamics is given in \cite{bungert2021}. 
From an optimization viewpoint, EKI can be interpreted as a derivative-free method, and subsequent work has further explored its gradient-flow structure \cite{Chada2019ConvergenceAO,tong2022localization,Weissmann_2022}. 
Applications in machine learning demonstrate its potential as a training method \cite{GSW2020,KS18}. 
Various regularization and stabilization strategies have also been proposed, including early stopping \cite{SS17b,iglesias2020adaptive}, iterative regularization \cite{Iglesias2015,2016InvPr..32b5002I}, and Tikhonov-regularized EKI (TEKI) \cite{CST2019,weissmann2021adaptive}. 
Furthermore, EKI can be interpreted as a low-rank approximation of Tikhonov regularization \cite{parzer2021convergence}. 
Extensions to constrained optimization have been studied in \cite{CSW19,hanu2023ensemble,Herty_2020}, and stochastic optimization variants include subsampling strategies \cite{Hanu_2023} and noisy dynamics \cite{weissmann2024ensemble}.

\subsection{Contributions}
The present work focuses on incorporating convex, possibly non-smooth regularization into deterministic EKI. While the gradient-flow structure of EKI and TEKI has been analyzed for smooth regularization, a corresponding theory for non-smooth subgradient dynamics has not yet been developed. The main contributions of this paper are summarized as follows.
\begin{itemize}
\item \textbf{Continuous-time formulation with non-smooth regularization.} 
We introduce a subgradient based continuous-time formulation of EKI for convex, lower semicontinuous regularization functionals. 
The resulting dynamics can be written as a system of differential inclusions for the ensemble mean coupled with ordinary differential equations for the ensemble deviations.
\item \textbf{Well-posedness of the particle dynamics.} In the linear forward-model setting, we establish well-posedness of the resulting particle system under minimal assumptions on the regularization functional. 
Existence and uniqueness of solutions are proved using Yosida approximations of maximal monotone operators. 
\item \textbf{Discrete-time convergence analysis.} 
Motivated by the continuous-time formulation, we propose an explicit discrete-time scheme that preserves the derivative-free structure of EKI. 
We prove convergence of the ensemble mean to the minimizer in the linear setting under strong convexity assumptions. 
The continuous-time dynamics serve primarily as a conceptual framework, while the convergence analysis is carried out directly for the implementable discrete-time algorithm.
\item \textbf{Numerical experiments and covariance freezing.} 
We evaluate the proposed framework on computed tomography and sparse recovery tasks involving total variation and $\ell_1$ regularization. To improve efficiency, we introduce a hybrid \emph{covariance freezing} strategy (Algorithm~\ref{alg:frozen_seki}), which fixes the ensemble covariance after a burn-in phase. This approach allows the SEKI dynamics to be interpreted as a mechanism for learning a problem-adapted preconditioner, reducing the per-iteration cost from $J$ forward evaluations to a single evaluation at the mean. In particular, this strategy reconciles the theoretical requirement $J-1\ge d$ with practical constraints, rendering the analysis-driven regime computationally tractable for high-dimensional problems.
\end{itemize}

\paragraph{Outline.}
The remainder of the paper is organized as follows.
\Cref{sec:mathsetup} introduces the mathematical setup for inverse problems and ensemble Kalman inversion.
\Cref{sec:wellposed} studies well-posedness of the continuous-time model, while \Cref{sec:discrete} analyzes its discrete-time formulation.
Numerical experiments are presented in \Cref{sec:numerics}, and concluding remarks are given in \Cref{sec:conclusion}.


\section{Mathematical setup}\label{sec:mathsetup}
In this section, we introduce the mathematical setting and notation used throughout the paper. 
The presentation follows the deterministic EKI framework, extended to allow for non-smooth regularization. Throughout this work, $\|\cdot\|$ denotes the Euclidean norm on $\R^d$. 
For a symmetric positive definite matrix $C\in\R^{d\times d}$, we define the weighted norm
$
\|x\|_{C} := \|C^{-1/2}x\| = \sqrt{x^\top C^{-1}x}\,.
$
We write $C([0,T];\R^d)$ for the space of continuous functions from $[0,T]$ to $\R^d$, and $L^2([0,T];\R^d)$ for the space of square-integrable measurable functions on $[0,T]$ with values in $\R^d$. 
For trajectories $x(\cdot)\in C([0,T];\R^{d\cdot J})$, the norm $\|\cdot\|_\infty$ denotes the supremum norm on $[0,T]$. 
For a symmetric positive semidefinite matrix $B$, we denote by $\lambda_{\min}(B)$ and $\lambda_{\max}(B)$ its smallest and largest eigenvalues, respectively. 
The identity matrix in $\R^{d\times d}$ is denoted by $I$.

\subsection{Inverse problem}
Consider the problem of recovering an unknown parameter $u \in \mathcal X$ from noisy observations of the form
\begin{equation}\label{eq:IP}
y = G(u) + \eta,
\end{equation}
where $\mathcal X$ is a real Hilbert space, $G:\mathcal X \to \mathbb R^{K}$ is a possibly nonlinear forward map, and $\eta \sim \mathcal N(0,\Gamma)$ denotes additive Gaussian noise with symmetric positive definite covariance matrix $\Gamma \in \mathbb R^{K \times K}$. In this work, we restrict ourselves to the finite-dimensional setting
\(
\mathcal X = \mathbb R^{d},
\)
and all norms are understood as Euclidean norms unless stated otherwise. Given observations $y \in \mathbb R^{K}$, define the variational formulation of the inverse problem
\begin{equation}\label{eq:variational_problem}
\min_{x \in \mathcal X} \; \Phi(x) + R(x),
\end{equation}
where the data misfit functional $\Phi:\R^d\to\R_+$ is defined as
\begin{equation}\label{eq:datamisfit}
\Phi(x) := \frac12 \|G(x)-y\|_\Gamma^2,
\end{equation}
and $R:\R^d \to (-\infty,+\infty]$ denotes a regularization functional. 
The composite objective functional is denoted as
\begin{equation}\label{eq:regularized_objective}
\Phi_R(x) := \Phi(x) + R(x).
\end{equation}

Throughout this work we assume that the regularization functional $R$ is convex and possibly non-smooth. 
More precisely, we impose the following standard assumption.

\begin{assumption}\label{ass:regularization}
The regularization functional $R:\R^d\to(-\infty,\infty]$ is proper, closed, convex, and lower semi-continuous. 
Moreover, we assume that $\inf_{x\in\R^d}R(x)>-\infty$.
\end{assumption}

Assumption~\ref{ass:regularization} is standard in convex regularization and is satisfied, for example, by total variation and $\ell^1$-type penalties, as well as by indicator functions of bounded convex constraint sets. Under this assumption, minimizers of \eqref{eq:variational_problem} exist and are characterized by first-order optimality conditions involving the subdifferential of $R$. 
Without loss of generality we further assume that $R(x)\ge0$ for all $x\in\R^d$.


\subsection{Ensemble Kalman inversion}\label{sec:EKI}
This section recalls deterministic EKI in its continuous-time formulation and its interpretation as a derivative-free optimization method. 
Based on this viewpoint, a subgradient-based extension suitable for non-smooth regularization will be introduced.

Let $J \ge 2$ denote the ensemble size. 
An ensemble consists of particles
\(
\{x^{(j)}\}_{j=1}^J \subset \mathbb R^{d},
\)
with ensemble mean $\bar x$ and empirical covariance $C$ defined by
\begin{equation}\label{eq:sample_covariance}
\bar x := \frac{1}{J}\sum_{j=1}^J x^{(j)},
\qquad
C(x) := \frac{1}{J}\sum_{j=1}^J (x^{(j)}-\bar x) (x^{(j)}-\bar x)^\top .
\end{equation}
The empirical cross-covariance is defined as
\begin{equation}\label{eq:cross_covariance}
C^{x,G}(x) :=
\frac{1}{J}\sum_{j=1}^J (x^{(j)}-\bar x)
\bigl(G(x^{(j)})-\bar G\bigr)^\top,
\qquad
\bar G := \frac{1}{J}\sum_{j=1}^J G(x^{(j)}).
\end{equation}
These empirical covariances provide derivative-free approximations of the linearized forward operator and play a central role in ensemble Kalman methods.

EKI arises from applying the EnKF to inverse problems. 
In its deterministic continuous-time formulation, EKI is described by the system of coupled ordinary differential equations
\begin{equation}\label{eq:cont_deterministicEKI}
\frac{{\mathrm d}x_t^{(j)}}{{\mathrm d}t}
=
C^{x,G}(x_t)\Gamma^{-1}\bigl(y - G(x_t^{(j)})\bigr),
\qquad j=1,\dots,J.
\end{equation}

The dynamic \eqref{eq:cont_deterministicEKI} admits a (derivative-free) optimization interpretation. 
For a linear forward map $G(x)=Ax$, the dynamic can be written as
\begin{equation}\label{eq:EKI_gradientflow}
\frac{{\mathrm d}x_t^{(j)}}{{\mathrm d}t}
=
- C(x_t)\nabla \Phi(x_t^{(j)}),
\end{equation}
corresponding to a preconditioned gradient flow whose time-dependent preconditioner is given by the empirical covariance matrix.

A fundamental structural property of EKI is the \emph{subspace property}: for fixed ensemble size $J$, all particles remain in the affine subspace spanned by the initial ensemble. 
Consequently, the method performs optimization restricted to this subspace. 
For nonlinear forward maps the representation \eqref{eq:EKI_gradientflow} holds only approximately. 
Nevertheless, EKI can still be viewed as a derivative-free approximation of a preconditioned gradient flow, where the approximation error depends on the ensemble spread \cite{Weissmann_2022}.

\subsection{Subgradient ensemble Kalman inversion}\label{sec:EKI_SEKI}
To stabilize the inversion and incorporate prior information, Tikhonov-regularized variants of EKI are commonly used \cite{CST2019}. 
The following section extends this idea to general convex, possibly non-smooth regularization functionals. 
In the smooth case, where $R$ is differentiable, the deterministic continuous-time formulation of regularized EKI reads
\begin{equation}\label{eq:cont_deterministicTEKI}
\frac{{\mathrm d}x_t^{(j)}}{{\mathrm d}t}
=
C^{x,G}(x_t)\Gamma^{-1}\bigl(y - G(x_t^{(j)})\bigr)
- C(x_t)\nabla R(x_t^{(j)}),
\qquad j=1,\dots,J.
\end{equation}
For quadratic Tikhonov regularization $R(x)=\frac\lambda2\|x\|^2$, $\lambda>0$, this reduces to a derivative-free approximation of the preconditioned gradient flow associated with the regularized objective $\Phi_R$ \cite{CST2019}.

The non-smooth case is considered next. Under \Cref{ass:regularization}, the gradient $\nabla R$ may not exist. 
Motivated by the optimality condition
\[
0 \in \nabla \Phi(x;y) + \partial R(x),
\]
gradients of $R$ in \eqref{eq:cont_deterministicTEKI} are replaced by subgradients. 
Let $R$ satisfy \Cref{ass:regularization}. 
The particle-wise subgradient ensemble Kalman inversion (SEKI) flow is defined by the coupled system of differential inclusions
\begin{equation}\label{eq:SEKI_particlewise}
\frac{{\mathrm d}x_t^{(j)}}{{\mathrm d}t}
\;\in\;
C^{x,G}(x_t)\Gamma^{-1}\bigl(y - G(x_t^{(j)})\bigr)
- C(x_t)\,\partial R(x_t^{(j)}),
\qquad j=1,\dots,J.
\end{equation}
Here, the set-valued term $C(x_t)\,\partial R(x_t^{(j)})$ is understood as
\[
C(x_t)\,\partial R(x_t^{(j)})
:= \{ C(x_t)\xi \mid \xi \in \partial R(x_t^{(j)}) \}.
\]
The dynamics \eqref{eq:SEKI_particlewise} provide a natural extension of regularized EKI to non-smooth objectives and can be interpreted as a derivative-free approximation of a preconditioned subgradient flow for minimizing $\Phi_R$. 
If $R$ is continuously differentiable, then $\partial R(x)=\{\nabla R(x)\}$ and \eqref{eq:SEKI_particlewise} reduces to the regularized EKI flow \eqref{eq:cont_deterministicTEKI}.

In \eqref{eq:SEKI_particlewise}, subgradients are evaluated for each particle separately. 
This leads to technical difficulties in the analysis of well-posedness, in particular in controlling the evolution of the sample covariance matrix. 
To address this issue, we consider the simplified coupled system of differential inclusions
\begin{equation}\label{eq:SEKI_cont}
\frac{{\mathrm d}x_t^{(j)}}{{\mathrm d}t}
\;\in\;
C^{x,G}(x_t)\Gamma^{-1}\bigl(y - G(x_t^{(j)})\bigr)
- C(x_t)\,\partial R(\bar x_t),
\qquad j=1,\dots,J\,.
\end{equation}
Here, subgradients are evaluated only at the ensemble mean $\bar x_t$. 
For both the theoretical analysis and the numerical experiments, we adopt the mean-subgradient formulation \eqref{eq:SEKI_cont}, which we refer to as the \emph{SEKI flow}.

To study the well-posedness, structural properties, and convergence behavior of the SEKI flow \eqref{eq:SEKI_cont}, we are interested in 
the ensemble deviations
\[
e^{(j)}(t) := x^{(j)}(t) - \bar x(t), \qquad j=1,\dots,J\,.
\]
The particle system \eqref{eq:SEKI_cont} then decomposes into the evolution of the ensemble mean and the ensemble deviations. 
The mean satisfies the differential inclusion
\begin{equation}\label{eq:SEKI_cont_mean}
\frac{{\mathrm d}\bar x(t)}{{\mathrm d}t}
+ C^{x,G}(t)\Gamma^{-1}(\bar G(t)-y)
\;\in\;
- C(t)\,\partial R(\bar x(t)),
\end{equation}
while the ensemble deviations evolve according to the ordinary differential equations
\begin{equation}\label{eq:spread_eq}
\frac{{\mathrm d} e^{(j)}(t)}{{\mathrm d}t}
=
- C^{x,G}(t)\Gamma^{-1}\bigl(G(x^{(j)}(t))-\bar G(t)\bigr).
\end{equation}

The empirical covariance matrix $C(t)$ satisfies the evolution equation
\begin{equation}\label{eq:covariance_eq}
\frac{{\mathrm d}C(t)}{{\mathrm d} t}
=
- C^{x,G}(t)\Gamma^{-1} C^{G,x}(t)
\end{equation}
where
\[
C^{G,x}(t)
=
\frac1J\sum_{j=1}^J
\bigl(G(x^{(j)}(t))-\bar G(t)\bigr) (x^{(j)}(t)-\bar x(t))^\top\,.
\]
From an optimization perspective, the resulting method can be interpreted as a derivative-free covariance-preconditioned subgradient method whose preconditioner is learned adaptively from the ensemble.

\subsection{Linear setting}
The analysis will focus on the linear forward model augmented with Tikhonov regularization. 
Following the standard TEKI formulation, this regularization can be incorporated through the augmented forward map $G:\R^d\to\R^{K+d}$ defined by
\[
G(x)= \begin{bmatrix}
    Ax \\
    x
\end{bmatrix}\in\R^{(K+d)\times d} , \quad \Gamma=\begin{bmatrix}
    \Gamma & \\
     & C_0
\end{bmatrix} \in\R^{(K+d)\times (K+d)},\quad y=\begin{bmatrix}
    y\\ 0
\end{bmatrix}\in\R^{K+d}
\]
where $A\in\R^{K\times d}$ and $C_0\in\R^{d\times d}$ is symmetric positive definite. 
In this setting, the mean–spread system simplifies considerably.

In particular, solutions of \eqref{eq:SEKI_cont} can equivalently be written in the form
\begin{equation}\label{eq:SEKI_flow_linear}
\begin{split}
    & \frac{{\mathrm d}\bar x(t)}{{\mathrm d}t}
      + C(t)A^\top\Gamma^{-1}(A\bar x(t)-y) + C(t)C_0^{-1}\bar x(t)
      \;\in\;
      - C(t)\,\partial R(\bar x(t)),\\
    & \frac{{\mathrm d} e^{(j)}(t)}{{\mathrm d}t}
      =
      - C(t)\bigl(A^\top\Gamma^{-1}A + C_0^{-1}\bigr)e^{(j)}(t),
      \qquad j=1,\dots,J\,.
\end{split}
\end{equation}

In this case the data misfit functional takes the form
\[
\Phi(x) := \frac12\|Ax-y\|_\Gamma^2 + \frac12\|x\|_{C_0}^2,
\qquad x\in\R^d.
\]
Instead of explicitly introducing Tikhonov regularization, an equivalent assumption would be that $A^\top\Gamma^{-1}A$ is strictly positive definite. 
The continuous-time formulation \eqref{eq:SEKI_flow_linear} provides a convenient framework to study structural properties of the dynamics, such as well-posedness, and to highlight the underlying covariance-preconditioned (sub)gradient-flow structure. From an optimization perspective, the mean dynamic corresponds to a descent method for the objective \eqref{eq:variational_problem}. 
However, since the empirical covariance acts as a time-dependent preconditioner, the continuous-time system does not uniquely determine a practical discretization. 
For this reason, convergence as an optimization method is analyzed directly for the discrete-time scheme introduced in \Cref{sec:discrete}. The following \Cref{sec:wellposed} establishes well-posedness and structural properties of the SEKI flow \eqref{eq:SEKI_flow_linear}.

\section{Well-posedness of the SEKI flow in the linear setting}\label{sec:wellposed}
This section establishes well-posedness of the particle system $(x^{(j)}(t),\, t\in[0,T])_{j=1,\dots,J}$ governed by the dynamics \eqref{eq:SEKI_flow_linear}. 
More precisely, existence and uniqueness of solutions on fixed time intervals $[0,T]$ are shown. 
The analysis relies on the Yosida approximation of the subdifferential, viewed as a maximal monotone set-valued operator. 
The proof follows the general strategy of Theorem~2.1 in Chapter~3 of \cite{aubin1984differential}. Throughout this section, the subdifferential $\partial R$ is therefore denoted by the maximal monotone operator $\mathcal A:\R^d\rightrightarrows\R^d$, which is admissible under \Cref{ass:regularization}. 
The only step of the proof that exploits the specific subdifferential structure of $\partial R$ is the derivation of uniform bounds for the Yosida approximation in \Cref{lem:Yosidabounded}. 
All remaining arguments apply to general maximal monotone operators.\medskip

We formulate the main result of this section:
\begin{theorem}[Well-posedness]\label{thm:main_existence}
    Let $G(\cdot)= A\cdot$ for some matrix $A\in\R^{K\times d}$ and suppose that \Cref{ass:regularization} is in place. Let $x^{(j)}(0) = x_0^{(j)}\in\R^d$, $j=1,\dots,J$ be such that the initial covariance $C(0)=C(x_0)$
    is positive definite. Then for every $T>0$ the system \eqref{eq:SEKI_flow_linear} admits a unique solution $(x^{(j)}(t),\, t\in[0,T])_{j=1,\dots,J}\in C([0,T]; \R^{d\cdot J})$.
\end{theorem}
The proof proceeds in several steps.

\begin{itemize}
\item First, the evolution of the empirical covariance matrix $C(\cdot)$ and its inverse $C^{-1}(\cdot)$ is characterized. 
\Cref{lem:cov} provides lower and upper bounds on the corresponding eigenvalues. 
Since the maximal monotone operator $\mathcal A$ is evaluated only at the ensemble mean $\bar x(t)$, the covariance evolution coincides with that of standard EKI.

\item Second, \Cref{prop:uniqueness} establishes uniqueness of solutions under the assumption that two solutions exist.

\item The main technical ingredient is the construction of the Yosida approximation. 
The maximal monotone operator is approximated via the resolvent operator $\mathcal J_\tau$, $\tau>0$. 
Replacing $\mathcal A$ by its Yosida approximation $\mathcal A_\tau$ yields a sequence of approximate solutions
\(
(x_{\tau}^{(j)}(t),\, t\in[0,T])_{j=1,\dots,J}.
\) \Cref{lem:Yosidabounded} provide uniform bounds on these approximate solutions.  

\item The uniform bounds on the approximate solutions imply that the sequence \(
(x_{\tau}^{(j)}(t),\, t\in[0,T])_{j=1,\dots,J}.
\) forms a Cauchy sequence in 
$C([0,T];\R^{d\cdot J})$ as stated in \Cref{lem:cauchy}. 
Its limit 
\[
(\hat x^{(j)}(t),\, t\in[0,T])_{j=1,\dots,J}
\in C([0,T];\R^{d\cdot J})
\]
is then shown to solve \eqref{eq:SEKI_flow_linear}.
\end{itemize}

\begin{remark}
    The statement of \Cref{thm:main_existence} implicitly assumes $J-1 \ge d$ to ensure that the empirical covariance matrix $C(t)$ remains invertible for all $t \ge 0$. This condition guarantees that the preconditioned dynamics are well-defined on the full space $\mathbb{R}^d$.

    In practice, ensemble Kalman methods evolve within the affine subspace spanned by the initial ensemble, whose dimension is at most $J-1$. Even when $J-1 < d$, the dynamics remain confined to this subspace due to the well-known subspace property of EKI. While the analysis could be reformulated in coordinates adapted to this ensemble subspace, where the covariance becomes invertible as an operator, a rigorous treatment of such a subspace-restricted formulation is left for future work.

    From an algorithmic perspective, in \Cref{sec:numerics} we propose \emph{covariance freezing} strategy (\Cref{alg:frozen_seki}) that provides a practical way to reconcile this theoretical requirement with computational efficiency. By assuming $J-1 \ge d$ during an initial burn-in phase (Phase I), the dynamics primarily serve to learn an informative covariance structure. Once this structure is frozen (Phase II), the method transitions to a preconditioned subgradient descent at the ensemble mean. This hybrid approach renders the theoretically motivated regime $J-1 \ge d$ computationally tractable for high-dimensional problems by reducing the per-iteration cost to a single forward evaluation. Additionally, in the regime $J-1 < d$, freezing the covariance helps to prevent further reduction in the search space through ensemble collapse.
\end{remark}

The remainder of this section is devoted to the proof of \Cref{thm:main_existence}.

\subsection{Sample covariance bounds}
The first step is to characterize the evolution of the empirical covariance matrix and to show that the ensemble collapses at rate $t^{-1}$. This includes both upper and lower bounds on the eigenvalues of the sample covariance.

\begin{lemma}[Sample covariance representation]\label{lem:cov}
Under the setting of \Cref{thm:main_existence}, let $(x^{(j)}(t),\, t\in[0,T])_{j=1,\dots,J}$ be a solution of \eqref{eq:SEKI_flow_linear} for some $T>0$. 
Then the empirical covariance matrices $(C(t),\ t\ge0)\subset\R^{d\times d}$ uniquely satisfy
\begin{equation}\label{eq:ODE_C}
\frac{{\mathrm d} C(t)}{{\mathrm d}t}
=
-2C(t)\bigl(A^\top \Gamma^{-1} A + C_0^{-1}\bigr) C(t).
\end{equation}
Moreover, for each $t\ge0$ the covariance matrix and its inverse admit the representations
\begin{equation}\label{eq:cov_solution}
C(t)
=
\Big(C^{-1}(0) + 2t (A^\top \Gamma^{-1} A + C_0^{-1})\Big)^{-1},
\qquad
C^{-1}(t)
=
C^{-1}(0) + 2t (A^\top \Gamma^{-1} A + C_0^{-1}).
\end{equation}

Both matrices are symmetric positive definite for all $t\ge0$ and depend continuously on $t$. 
Furthermore, there exist constants $\ell_1,\ell_2,u_1,u_2>0$ such that the eigenvalues
\(
\lambda_{\min}(t)=\lambda_1(t)\le\dots\le\lambda_d(t)=\lambda_{\max}(t)
\)
of $C(t)$ satisfy
\begin{equation}\label{eq:eigenvalue_bounds}
\frac{\ell_1}{\ell_2(t+1)}
\le
\lambda_{\min}(t)
\le
\lambda_{\max}(t)
\le
\frac{u_1}{u_2(t+1)}
\end{equation}
for all $t\ge0$.
\end{lemma}

\begin{proof}
The result follows directly from Theorem~3.3 and Lemma~3.12 in \cite{CST2019}.
\end{proof}

\subsection{Uniqueness of solutions}
Using the characterization of the empirical covariance, uniqueness of solutions to \eqref{eq:SEKI_flow_linear} can be established.

\begin{proposition}[Uniqueness of solutions]\label{prop:uniqueness}
    Under the setting of \Cref{thm:main_existence}, let $(x^{(j)}(t),\ t\ge0)_{j=1,\dots,J}$ and $(z^{(j)}(t),\ t\ge0)_{j=1,\dots,J}$ be two solutions of \eqref{eq:SEKI_flow_linear} initialized with $x^{(j)}(0) = x_0^{(j)}\in\R^d$ and $z^{(j)}(0)=z_0^{(j)}\in\R^d$, $j=1,\dots,J$, such that \[\frac1J\sum_{j=1}^J(x_0^{(j)}-\bar x_0)(x_0^{(j)}-\bar x_0)^\top = \frac1J\sum_{j=1}^J(z_0^{(j)}-\bar z_0)(z_0^{(j)}-\bar z_0)^\top\,. \]
    Then for all $t\ge0$,
    \[ C^{x}(t) = \frac1J\sum_{j=1}^J(x^{(j)}(t)-\bar x(t))(x^{(j)}(t)-\bar x(t))^\top = \frac1J\sum_{j=1}^J(z^{(j)}(t)-\bar z(t))(z^{(j)}(t)-\bar z(t))^\top = C^{z}(t)\]
    and the ensemble means satisfy the contraction estimate    
    \[\|\bar x(t) - \bar z(t)\|_{C(t)}\le \|\bar x_0 - \bar z_0\|_{C(0)}\,,\]
    where $C(t) = C^{x}(t) = C^{z}(t)$.
\end{proposition}

\begin{proof}
    
    The equality $C^x(t)=C^z(t)$ for all $t\ge0$ follows from the uniqueness of solutions to the covariance equation \eqref{eq:ODE_C}. Moreover, by \Cref{lem:cov}, the covariance matrix $C(t)$ is invertible for all $t\ge0$. Since $\mathcal A$ is maximal monotone, it follows
    \begin{align*}
        0&\ge \Bigg\langle \bar x(t) - \bar z(t), C^{-1}(t) \Big(\frac{{\mathrm d}\bar x(t)}{{\mathrm d}t} + C(t)(A^\top \Gamma^{-1}(A\bar x(t)-y) + C_0^{-1}\bar x(t) \Big) \\ & \qquad\qquad\qquad\ -C^{-1}(t) \Big(\frac{{\mathrm d}\bar z(t)}{{\mathrm d}t} + C(t)(A^\top \Gamma^{-1}(A\bar z(t)-y) + C_0^{-1}\bar z(t) \Big)\Bigg\rangle  \\
        &=\Big\langle C^{-1}(t)(\bar x(t) - \bar z(t)),\frac{{\mathrm d}\bar x(t)}{{\mathrm d}t}-\frac{{\mathrm d}\bar z(t)}{{\mathrm d}t} \Big\rangle + \Big\langle \bar x(t)-\bar z(t) , (A^\top \Gamma^{-1}A+ C_0^{-1})(\bar x(t)-\bar z(t))\Big\rangle\,.
    \end{align*}
    By chain rule it holds
    \begin{align*}
        \frac{{\mathrm d}\|\bar x(t)-\bar z(t)\|_{C(t)}^2}{{\mathrm d}t} &= \frac{{\mathrm d}(\bar x(t)-\bar z(t))^\top C^{-1}(t) (\bar x(t)-\bar z(t))}{{\mathrm d}t}\\ &= 2\Big\langle \bar x(t)-\bar z(t), C^{-1}(t)\big(\frac{{\mathrm d}\bar x(t)}{{\mathrm d}t}-\frac{{\mathrm d}\bar z(t)}{{\mathrm d}t} \big)\Big\rangle\\
        &\quad-\Big\langle \bar x(t)-\bar z(t) , C^{-1}(t)\frac{{\mathrm d}C(t)}{{\mathrm d}t} C^{-1}(t)\big(\bar x(t)-\bar z(t)\big)\Big\rangle\\
        & = 2\Big\langle C^{-1}(t)\big(\bar x(t)-\bar z(t)\big), \frac{{\mathrm d}\bar x(t)}{{\mathrm d}t}-\frac{{\mathrm d}\bar z(t)}{{\mathrm d}t} \Big\rangle\\
        &\quad+2\langle C^{-1}(t)\big(\bar x(t)-\bar z(t)\big), C(t)(A^\top\Gamma^{-1}A+C_0^{-1})C(t)C^{-1}(t)\big(\bar x(t)-\bar z(t)\big)\Big\rangle\\
        & = 2\Big\langle C^{-1}(t)\big(\bar x(t)-\bar z(t)\big), \frac{{\mathrm d}\bar x(t)}{{\mathrm d}t}-\frac{{\mathrm d}\bar z(t)}{{\mathrm d}t} \Big\rangle\\
        &\quad+2\langle \big(\bar x(t)-\bar z(t)\big), (A^\top\Gamma^{-1}A+C_0^{-1})\big(\bar x(t)-\bar z(t)\big)\Big\rangle \le 0\,,
    \end{align*}
    where have used that $C$ is given by the ODE \eqref{eq:ODE_C}.
\end{proof}
\begin{remark}
    The uniqueness of the full particle \eqref{eq:SEKI_flow_linear} follows from the fact that we can bound the ensemble deviations by
    \[\|x^{(j)}(t)-\bar x(t) - (z^{(j)}(t)-\bar z(t))\| \le \|x_0^{(j)}-\bar x_0 - (z_0^{(j)}-\bar z_0)\|\,,  \]
    implying that
    \[\|x^{(j)}(t)-z^{(j)}(t) \| \le \|x_0^{(j)}-\bar x_0 - (z_0^{(j)}-\bar z_0)\| + \|\bar x_0-\bar z_0\|\,.\]
\end{remark}

\subsection{Yosida approximation}
The next step in the proof introduces the Yosida approximation $\mathcal A_\tau$ of the maximal monotone operator $\mathcal A:\R^d\rightrightarrows\R^d$, defined by
\[
\mathcal A_{\tau} := \frac{1}{\tau} (\mathrm{Id}-\mathcal J_\tau),\qquad \tau >0,
\]
where $\mathcal J_\tau:=({\mathrm{Id}}+\tau \mathcal A)^{-1}$ denotes the resolvent operator. 
In the present setting, where $\mathcal A=\partial R$, the resolvent is given by the proximal operator
\[
\mathcal J_\tau(x) := \mathrm{prox}_{\tau R}(x)
:= \arg\min_{z\in\R^d}\left\{R(z)+\frac1{2\tau}\|z-x\|^2\right\},\qquad x\in\R^d.
\]
The associated Moreau envelope is defined by
\[
R_\tau(x):=\min_{z\in\R^d}\left\{R(z)+\frac1{2\tau}\|x-z\|^2\right\},\qquad x\in\R^d,
\]
and satisfies $\nabla R_\tau = \mathcal A_\tau$. 
The following properties hold \cite[Theorem~1.2, Chapter~3]{aubin1984differential}:
\begin{itemize}
    \item $\mathcal A_\tau(x)\in \mathcal A(\mathcal J_\tau x)$ for all $x\in\R^d$ and $\tau>0$,
    \item $\mathcal A_\tau$ is $\frac1\tau$-Lipschitz continuous,
    \item $\mathcal J_\tau x\to x$ as $\tau\to0$ for all $x\in\R^d$.
\end{itemize}

Let $x_\tau = (x_\tau^{(j)}(t),\, t\in[0,T])_{j=1,\dots,J}$ denote the unique solution of the regularized system
\begin{equation}\label{eq:Yosida_approx_linear}
\begin{split}
\frac{{\mathrm d}\bar x_\tau(t)}{{\mathrm d}t}
&+
C_\tau(t)A^\top\Gamma^{-1}(A \bar x_\tau(t)-y)
+
C_\tau(t)\mathcal A_\tau(\bar x_\tau(t))
= 0,\\
\frac{{\mathrm d} e_\tau^{(j)}(t)}{{\mathrm d}t}
&=
- C_\tau(t)(A^\top\Gamma^{-1} A +C_0^{-1})e_\tau^{(j)}(t),
\qquad j=1,\dots,J,
\end{split}
\end{equation}
where
\[
C_\tau(t) = \frac{1}{J}\sum_{j=1}^J e_\tau^{(j)}(t)(e_\tau^{(j)}(t))^\top,
\qquad
e_\tau^{(j)}(t) = x_\tau^{(j)}(t)-\bar x_\tau(t),
\qquad
\bar x_\tau(t)=\frac1J\sum_{j=1}^J x_\tau^{(j)}(t).
\]
As in the original system, the covariance satisfies $C_\tau(t)=C(t)$ and is therefore given by \eqref{eq:cov_solution}.

The first step is to establish uniform boundedness of $\bar x_\tau(\cdot)$ with respect to $\tau>0$, together with an $L^2$-bound for the Yosida approximation.

\begin{lemma}[Uniform boundedness]\label{lem:Yosidabounded}
    Under the setting of \Cref{thm:main_existence}, for any time horizon $T>0$, initialization $x_0 \in\R^{d\cdot J}$, $\tau_{\max}>0$ and $\tau\in(0,\tau_{\max}]$ such that
    \(C(0)= \frac1J\sum_{j=1}^J(x_0^{(j)}-\bar x_0)(x_0^{(j)}-\bar x_0)^\top\)
    is positive definite
    there exists a unique solution $(x_\tau^{(j)}(t),\, t\in[0,T])_{j=1,\dots,J}$ of \eqref{eq:Yosida_approx_linear} with $x_\tau(0) = x_0$ that is continuously differentiable. Moreover, for all $t\in[0,T]$ this solution satisfies
    \[\|\bar x_\tau(t)\|^2 \le R_1\,\quad \text{and}\quad \int_{0}^t\|\cA_\tau(\bar x_\tau(s))\|^2\,\dd s\le R_2\,, \]
    for constants $R_1\, , R_2>0$ independent of $\tau\in(0,\tau_{\max}]$ and all $t\in[0,T]$. The constant $R_1$ is also independent of $T$.
\end{lemma}
\begin{proof}
    The existence of local unique solutions of \eqref{eq:Yosida_approx_linear} follows by local Lipschitz continuity of the drift term. The global extensions of these solutions can be deduced by the following Step 1.\medskip

    \textit{\underline{Step 1:} Uniform boundedness of the trajectories.} 
    Let $(x_\tau^{(j)}(t),\, t\in[0,T])_{j=1,\dots,J}$ be some solution of \eqref{eq:Yosida_approx_linear}.
    Let $x^\ast\in\R^d$ with 
    \[0\in \nabla \Phi(x^\ast) + \partial R(x^\ast)\,,\]
    or equivalently 
    \[ u^\ast := -\nabla \Phi(x^\ast)\in \partial R(x^\ast)\,.\]
    Next, define 
    \[ x_\tau^\ast:=x^\ast-\tau u^\ast\,.\]
    Note that any $z=\cJ_\tau(x)$ can be equivalently characterized by $x-z\in\tau\, \partial R(z)$. Since $x_\tau^\ast-x^\ast = \tau u^\ast \in \tau\, \partial R(x^\ast)$, we have that $x^\ast =\cJ_\tau(x_\tau^\ast)$, and, therefore,
    \[\cA_\tau(x_\tau^\ast) = \frac1\tau (x_\tau^\ast-\cJ_\tau(x_\tau^\ast)) = u^\ast\,.\]
    For $V_\tau:\R_+\to\R_+$ defined as $V_\tau(t):=\frac12\|\bar x_\tau(t)-x_\tau^\ast\|_{C(t)}^2$
    it holds
    \begin{align*}
        \frac{\dd V_\tau(t)}{\dd t} &= \langle C^{-1}(t)(\bar x_\tau(t)-x_\tau^\ast),\frac{\dd \bar x_\tau(t)}{\dd t}\rangle + \frac12\langle C^{-1}(t)\frac{\dd C(t)}{\dd t} C^{-1}(t)(\bar x_\tau(t)-x_\tau^\ast),\bar x_\tau(t)-x_\tau^\ast\rangle \\
        &=-\langle \bar x_\tau(t)-x_\tau^\ast,\cA_\tau(\bar x_\tau(t))-\cA_\tau(x_\tau^\ast)\rangle\le 0\,.
    \end{align*}
    This verifies that $V_\tau(t)\le V_\tau(0)\le 2\|\bar x_\tau(0)-x^\ast\|^2 + 2\tau^2\|u^\ast\|^2$. Moreover, we deduce that
    \begin{align*}
        \|\bar x_\tau(t)-x^\ast\|^2 \le 2\|\bar x_\tau(t)-x_\tau^\ast\|^2+2\|x_\tau^\ast-x^\ast\|^2 &\le \frac{2u_1}{u_2(t+1)} V_\tau(0) + 2\tau^2\|u^\ast\|^2\\
        &\le \frac{4u_1}{u_2} \|\bar x_\tau(0)-x^\ast\|^2 + 2(1+2\frac{u_1}{u_2})\tau^2 \|u^\ast\|^2\,.
    \end{align*}
    Hence, for any solution of \eqref{eq:Yosida_approx_linear} there exists $R_1$ independent of $\tau$ and $T$ such that
    \[ \|\bar x_\tau(t)\|^2\le R_1\,.\]\medskip

    \textit{\underline{Step 2:} Uniform $L^2$-bound for the Yosida approximation.} For the second claim define the energy
    \[E_\tau(t) = \Phi(\bar x_\tau(t)) + R_\tau(\bar x_\tau(t))\]
    where $R_\tau (x) = \min_{z\in\R^d} R(z) + \frac{1}{2\tau} \|z-x\|^2\le R(x)$ which satisfies $\cA_\tau = \nabla R_\tau$. It holds true that
    \[\frac{\dd E_\tau(t)}{\dd t} = -\langle \nabla E_\tau(t),C(t)\nabla E_\tau (t)\rangle\le 0 \]
    and, hence,
    \[\int_{0}^T\frac{\ell_1}{\ell_2(s+1)} \|\nabla E_\tau(s)\|^2\,\dd s\le E_\tau(0)\,.\]
    This implies that for all $t\in[0,T]$ it holds 
    \begin{align*}
        \int_0^t\|\cA_\tau(\bar x_\tau(s)\|^2\,\dd s&\le \frac{2\ell_2(T+1)}{\ell_1}\int_0^{T}\frac{\ell_1}{\ell_2(s+1)}\|\nabla E_\tau(s)\|^2\,\dd s+2\int_0^T\|\nabla \Phi(\bar x_\tau(s))\|^2\,\dd s\\
        &\le \frac{2\ell_2(T+1)}{\ell_1}E_\tau(0) + 4L^2\int_0^T\|\bar x_\tau(s)-x_\ast\|^2\,\dd s+4T\|\nabla\Phi(x^\ast)\|^2\,.
    \end{align*}
    Since $\|\bar x_\tau(s)-x_\tau^\ast\|^2$ is uniformly bounded (both in $T$ and $\tau$), there exists $R_2>0$ depending on $T$ but independent of $\tau\in(0,\tau_{\max}]$ such that 
    \[\int_0^t\|\cA_\tau(\bar x_\tau(s)\|^2\,\dd s\le R_2\]
    for all $t\in[0,T]$.
\end{proof}

The following lemma shows that the family $\{(x_\tau(t)^{(j)},\,t\in[0,T])_{j=1,\dots,J}\}_{\tau>0}$ forms a Cauchy sequence in $C([0,T]; \R^{d\cdot J})$ and therefore converges uniformly on $[0,T]$ to a continuous limit $(\hat x^{(j)}(t),j=1,\dots,J)_{t\ge0}$ as $\tau\to0$.
This limit will serve as a candidate solution of the original system \eqref{eq:SEKI_flow_linear}.
\begin{lemma}[Cauchy sequence]\label{lem:cauchy}
    Under the setting of \Cref{thm:main_existence}, for $T>0$ let $x_\tau = (x_\tau^{(j)}(t),\, t\in[0,T])_{j=1,\dots,J}$ and $x_\mu = (x_\mu^{(j)},\, t\in[0,T])_{j=1,\dots,J}$ be the unique solutions of \eqref{eq:Yosida_approx_linear} with common initialization $x_\tau(0) = x_\mu(0) = x_0\in\R^{d\cdot J}$ and parameters $\tau,\mu>0$. Then for all $t\in[0,T]$ and $j=1,\dots,J$,
    \[\|x_\tau^{{(j)}}(t)-x_\mu^{(j)}(t)\|^2 \le \frac{T\, R_2^2\, u_1}{4\,u_2} \cdot (\tau+\mu)\,, \]
    where $R_2>0$ is the constant from \Cref{lem:Yosidabounded}, independent of $\tau,\mu\in(0,\tau_{\max}]$. In particular, for any sequence $(\tau_n)_{n\in\mathbb N}$ with $\tau_n\to 0$, the sequence $\{(x_{\tau_{n}}^{(j)}(t),\, t\in[0,T])_{j=1,\dots,J},n\in\N \}\subset C([0,T]; \R^{d\cdot J})$ is a Cauchy sequence in $C([0,T]; \R^{d\cdot J})$ with respect to the $\|\cdot\|_\infty$-norm.
\end{lemma}

\begin{proof}
    For $T>0$, let $x_\tau = (x_\tau^{(j)},\, t\in[0,T])_{j=1,\dots,J}$ and $x_\mu = (x_\mu^{(j)},\, t\in[0,T])_{j=1,\dots,J}$ be the unique solutions of \eqref{eq:Yosida_approx_linear} with initialization $x_\tau(0) = x_\mu(0) = x_0\in\R^{d\cdot J}$ and arbitrary $\tau,\mu\in(0,\tau_{\max}]$ for arbitrary $\tau_{\max}>0$.
    Since the covariance evolution is independent of $\tau$, it holds that 
    \(
    C_\tau(t)=C_\mu(t)=C(t)
    \)
    for all $t\in[0,T]$, where $C(t)$ is given by \eqref{eq:cov_solution}.

    By chain rule, we have
    \begin{align*}
        \frac{{\mathrm d}\frac12\|\bar x_\tau(t)-\bar x_\mu(t)\|_{C(t)}^2}{{\mathrm d}t} &= \Big\langle \bar x_\tau(t)-\bar x_\mu(t), C^{-1}(t) \big(\frac{{\mathrm d} \bar x_\tau(t)}{{\mathrm d}t} - \frac{{\mathrm d} x_\mu(t)}{{\mathrm d}t}\big)\Big\rangle\\
        &\quad + \frac12\Big\langle C^{-1}(t)(\bar x_\tau(t)-\bar x_\mu(t)), \frac{{\mathrm d} C(t)}{{\mathrm d} t} C^{-1}(t)(\bar x_\tau(t)-\bar x_\mu(t)) \Big\rangle\\
        &= \Big\langle C^{-1}(t)\big(\bar x_\tau(t)-\bar x_\mu(t)\big),  \frac{{\mathrm d} \bar x_\tau(t)}{{\mathrm d}t} - \frac{{\mathrm d} x_\mu(t)}{{\mathrm d}t}\Big\rangle\\
        & \quad + \Big\langle (\bar x_\tau(t)-\bar x_\mu(t)), (C_0^{-1} + A^\top \Gamma^{-1} A)(\bar x_\tau(t)-\bar x_\mu(t)) \Big\rangle\\
        & = -\langle \bar x_\tau(t)-\bar x_\mu(t),\cA_\tau(\bar x_\tau(t))-\cA_\mu(\bar x_\mu(t))\rangle \,.
    \end{align*}
    Using the initial condition $x_\tau(0) = x_\mu(0)$ the above differential equation can be written in integral form
    \begin{align*}
        &\frac12 \|\bar x_\tau(t)-\bar x_\mu(t)\|_{C(t)}^2\\ & = -\int_0^t \langle \bar x_\tau(s)-\bar x_\mu(s),\cA_\tau(\bar x_\tau(s)) - \cA_\mu (\bar x_\mu (s))\rangle\,{\mathrm d}s\\
        &=-\int_0^t \langle \tau \cA_\tau (\bar x_\tau(s)) + \cJ_\tau (\bar x_\tau(s)) - \mu\cA_\mu(\bar x_\mu(s))-\cJ_\mu(\bar x_\mu(s)),\cA_\tau(\bar x_\tau(s)) - \cA_\mu (\bar x_\mu (s))\rangle\,{\mathrm d}s\\
        &= -\int_0^t \langle \tau \cA_\tau(\bar x_\tau(s))- \mu\cA_\mu(\bar x_\mu(s)),\cA_\tau(\bar x_\tau(s)) - \cA_\mu (\bar x_\mu (s))\rangle\,{\mathrm d}s\\
        &\, \quad -\int_0^t \langle \cJ_\tau (\bar x_\tau(s))-\cJ_\mu(\bar x_\mu(s)),\cA_\tau(\bar x_\tau(s)) - \cA_\mu (\bar x_\mu (s))\rangle\,{\mathrm d}s\\
        &\le -\int_0^t \langle \tau \cA_\tau(\bar x_\tau(s))- \mu\cA_\mu(\bar x_\mu(s)),\cA_\tau(\bar x_\tau(s)) - \cA_\mu (\bar x_\mu (s))\rangle\,{\mathrm d}s
    \end{align*}
    where we used $\langle \cJ_\tau (\bar x_\tau(s))-\cJ_\mu(\bar x_\mu(s)),\cA_\tau(\bar x_\tau(s)) - \cA_\mu (\bar x_\mu (s))\rangle\ge0$ since $\cA_\tau(\bar x_\tau(s)) \in \cA(J_\tau(\bar x_\tau(s))$ and $\cA_\mu(\bar x_\mu(s)) \in \cA(J_\mu(\bar x_\mu(s))$ for all $s\ge0$. The last expression can be further rewritten as
    \begin{align*}
        -\int_0^t \langle \tau \cA_\tau(\bar x_\tau(s))&- \mu\cA_\mu(\bar x_\mu(s)),\cA_\tau(\bar x_\tau(s)) - \cA_\mu (\bar x_\mu (s))\rangle\,{\mathrm d}s\\ &= \tau \int_0^t \langle \cA_\tau(\bar x_\tau(s)),\cA_\mu(\bar x_\mu(s))\rangle\, {\mathrm d}s +  \mu \int_0^t \langle \cA_\mu(\bar x_\mu(s)),\cA_\tau(\bar x_\tau(s))\rangle\, {\mathrm d}s \\
        &\quad -\tau\int_0^t\|\cA_\tau(\bar x_\mu(s))\|^2\,{\mathrm d}s -\mu\int_0^t\|\cA_\mu(\bar x_\mu(s))\|^2\,{\mathrm d}s\,.
    \end{align*}
    Using Cauchy-Schwarz inequality followed by Young's inequality implies
    \[ \tau \langle\cA_\tau(\bar x_\tau(s)),\cA_\mu(\bar x_\mu(s))\rangle \le \tau \|\cA_\tau(\bar x_\tau(s))\|^2 + \frac{\tau}{4} \|\cA_\mu(\bar x_\mu(s))\|^2\]
    and similarly, 
    \[ \mu \langle\cA_\mu(\bar x_\mu(s)),\cA_\tau(\bar x_\tau(s))\rangle \le \mu \|\cA_\mu(\bar x_\mu(s))\|^2 + \frac{\mu}{4} \|\cA_\tau(\bar x_\tau(s))\|^2\,.\]
    By \Cref{lem:Yosidabounded} we deduce that
    \begin{align*}
        \frac12 \|\bar x_\tau(t)-\bar x_\mu(t)\|_{C(t)}^2\le t \frac{\tau+\mu}{4} R_2^2
    \end{align*}
    where $R_2>0$ is independent of $\tau$ and $\mu$. By \Cref{lem:cov} the largest eigenvalue of $C(t)$ satisfies 
    \[ \lambda_{\max}(t)\le \frac{u_1}{u_2(t+1)}\]
    such that
    \[ \frac12\|\bar x_\tau(t)-\bar x_\mu(t)\|^2 \le \frac{\lambda_{\max}(t)}2 \|\bar x_\tau(t)-\bar x_\mu(t)\|_{C(t)}^2 \le \frac{T\, R_2^2\, u_1}{u_2} \cdot \frac{\tau+\mu}{4}\]
    proving the first assertion. For the second claim, we note that for any $t\in [0,T]$, $j\in\{1,\dots,J\}$ and $\mu,\tau>0$ it holds $\|x_\tau^{(j)}(t)-\bar x_\tau(t)-(x_\mu^{(j)}-\bar x_\mu(t))\|=0$, so that
    \begin{align*}
        \frac{1}4\|x_\tau^{(j)}(t) - x_\mu^{(j)}(t)\|^2 &\le \frac12 \|x_\tau^{(j)}(t)-\bar x_\tau(t)-(x_\mu^{(j)}-\bar x_\mu(t))\|^2 + \frac{1}2\|\bar x_\tau(t)-\bar x_\mu(t)\|^2\\
        &= \frac{1}2\|\bar x_\tau(t)-\bar x_\mu(t)\|^2 \le \frac{T\, R_2^2\, u_1}{u_2} \cdot \frac{\tau+\mu}{4}\,.
    \end{align*}
    This proves the second assertion.
\end{proof}

\subsection{Proof of \Cref{thm:main_existence}}
We are now in a position to prove \Cref{thm:main_existence} by passing to the limit $\tau \to 0$ in the Yosida approximation.
\begin{proof}[Proof of \Cref{thm:main_existence}]
    Let $(\tau_n)_{n\in\mathbb N}$ be a sequence with $\tau_n \downarrow 0$, and let 
    \(
    x_{\tau_n} = (x_{\tau_n}^{(j)}(t),\, t\in[0,T])_{j=1,\dots,J}
    \)
    denote the corresponding solutions of \eqref{eq:Yosida_approx_linear}.
    \medskip
    
    \textit{\underline{Step 1:} Strong convergence of the trajectories.}  
    By \Cref{lem:cauchy}, the sequence $(x_{\tau_n})_{n\in\mathbb N}$ is Cauchy in $C([0,T];\R^{d\cdot J})$ and therefore converges uniformly to some limit
    \[
    \hat x = (\hat x^{(j)}(t),\, t\in[0,T])_{j=1,\dots,J} \in C([0,T];\R^{d\cdot J}).
    \]
    In particular, the convergence also holds strongly in $L^2([0,T];\R^{d\cdot J})$.
    \medskip
    
    \textit{\underline{Step 2:} Convergence of the resolvent.}
    By definition of the Yosida approximation,
    \[
    x_\tau - \mathcal J_\tau(x_\tau) = \tau \mathcal A_\tau(x_\tau).
    \]
    Using \Cref{lem:Yosidabounded}, this implies
    \[
    \int_0^T \|x_{\tau_n}^{(j)}(t) - \mathcal J_{\tau_n}(x_{\tau_n}^{(j)}(t))\|^2 \,{\mathrm d}t
    \le \tau_n R_2 \to 0.
    \]
    Hence,
    \[
    \mathcal J_{\tau_n}(x_{\tau_n}^{(j)}(\cdot)) \to \hat x^{(j)}(\cdot)
    \quad \text{strongly in } L^2([0,T];\R^d)\,.
    \]
    The same argument applies to the ensemble mean:
    \[
    \mathcal J_{\tau_n}(\bar x_{\tau_n}(\cdot)) \to \bar{\hat x}(\cdot)
    \quad \text{in } L^2([0,T];\R^d)\,.
    \]
    \medskip
    
    \textit{\underline{Step 3:} Convergence of the smooth part.}
    By continuity of $\nabla \Phi$, it follows that
    \[
    \nabla \Phi(x_{\tau_n}^{(j)}(\cdot))
    \to
    \nabla \Phi(\hat x^{(j)}(\cdot))
    \quad \text{strongly in } L^2([0,T];\R^d)\,.
    \]
    \medskip
    
    \textit{\underline{Step 4:} Weak convergence of the subgradients.}
    \Cref{lem:Yosidabounded} states that
    \[
    \mathcal A_{\tau_n}(\bar x_{\tau_n}(\cdot))
    \]
    is bounded in $L^2([0,T];\R^d)$. Hence, there exists a subsequence (not relabeled) and a limit $w \in L^2([0,T];\R^d)$ such that
    \[
    \mathcal A_{\tau_n}(\bar x_{\tau_n}(\cdot)) \rightharpoonup w(\cdot)
    \quad \text{weakly in } L^2([0,T];\R^d).
    \]
    Since $A_{\tau_n}(\bar{x}_{\tau_n}) \in A(J_{\tau_n}(\bar{x}_{\tau_n}))$, 
$J_{\tau_n}(\bar{x}_{\tau_n}) \to \hat{\bar{x}}$ strongly, and 
$A_{\tau_n}(\bar{x}_{\tau_n}) \rightharpoonup w$ weakly in $L^2$, 
the strong--weak closedness of maximal monotone operators 
\cite[Proposition~1.2, Chapter~3]{aubin1984differential} yields
\[
w(t) \in A(\hat{\bar{x}}(t)) \quad \text{for a.e. } t \in [0,T]\,.
\]
    \medskip
    
    \textit{\underline{Step 5:} Weak compactness of time derivatives.}
    by \Cref{lem:Yosidabounded}, we have that 
    \begin{align*} 
        \Big\|\frac{{\mathrm d} x_\tau^{(j)}(t)}{{\mathrm d}t}\Big\| &\le \Big\|\frac{{\mathrm d} (x_\tau^{(j)}(t)-\bar x_\tau(t))}{{\mathrm d}t}\Big\| + \Big\|\frac{{\mathrm d} \bar x_\tau(t)}{{\mathrm d}t}\Big\|\\ &\le \|C(t)\|^2 \|A^\top \Gamma^{-1}A+C_0^{-1}\| +  \Big\|\frac{{\mathrm d} \bar x_\tau(t)}{{\mathrm d}t}\Big\| \\ &\le \lambda_{\max}^2(t)\|A^\top \Gamma^{-1}A+C_0^{-1}\|+\lambda_{\max}(t) \Big(\|A_{\tau}(\bar x_\tau(t))\| + \|A^\top \Gamma^{-1} (A\bar x_\tau(t)-y)\| \Big) 
    \end{align*}
    so that 
    \[ 
        \int_{0}^T \Big\|\frac{{\mathrm d} x_\tau^{(j)}(t)}{{\mathrm d}t}\Big\|^2\,\dd t \le 4T\frac{u_1^2}{u_2^2}+4\frac{u_1}{u_2} \Big( R_2 + T(\|A^\top \Gamma^{-1} y\|^2 + R_1 \|A^\top \Gamma^{-1} A\|^2 )\Big) 
    \]
    and, therefore, $\{ (\frac{{\mathrm d}x_{\tau_n}^{(j)}(t)}{{\mathrm d}t},\, t\in[0,T])_{j=1,\dots,J},n\in\N\}$ is bounded in $L^2([0,T];\R^{d\cdot J})$. 
    Hence, there exists a sub-sequence $\{ (\frac{{\mathrm d}x_{\tau_{n_k}}^{(j)}(t)}{{\mathrm d}t},\, t\in[0,T])_{j=1,\dots,J},k\in\N\}$ that converges weakly to some $(\frac{{\mathrm d}\tilde x^{(j)}(t)}{{\mathrm d}t},\, t\in[0,T])_{j=1,\dots,J}$ in $L^2([0,T];\R^{d\cdot J})$. By strong convergence of $\{(x_{\tau_{n}}^{(j)}(t),\, t\in[0,T])_{j=1,\dots,J},n\in\N \}\subset C([0,T]; \R^{d\cdot J})$ in $L^2([0,T];\R^{d\times J})$, we have that
    $(\frac{{\mathrm d}\tilde  x^{(j)}(t)}{{\mathrm d}t})_{j=1,\dots,J} = (\frac{{\mathrm d}\hat  x^{(j)}(t)}{{\mathrm d}t})_{j=1,\dots,J}$ for almost all $t\in[0,T]$.
    \medskip
    
    \textit{\underline{Step 6:} Passage to the limit.} 
    Passing to the limit $k\to\infty$, we obtain
    \begin{align*}
        0 = \frac{{\mathrm d}x_{\tau_{n_k}}^{(j)}(\cdot)}{\mathrm {d}t} + C(t) \nabla \Phi({x_{\tau_{n_k}}^{(j)}(\cdot)) + C(t) \cA_{\tau_{n_k}}(\bar x_{\tau_{n_k}}(\cdot))} \rightharpoonup \frac{{\mathrm d}\hat x^{(j)}(\cdot)}{\mathrm {d}t} + C(t) \nabla \Phi(\hat x^{(j)}(\cdot)) + C(t) w(\cdot)
    \end{align*}
    with $w(t)\in\cA(\bar{\hat x}(t))$ for almost every $t\in[0,T]$. Hence, 
    \[\frac{{\mathrm d}\hat x^{(j)}(t)}{\mathrm {d}t} + C(t) \nabla \Phi(\hat x^{(j)}(t)) + C(t) w(t) = 0\]
    for almost every $t\in[0,T]$ meaning $(\hat x^{(j)}(t),\, t\in[0,T])_{j=1,\dots,J}\in C([0,T]; \R^{d\cdot J})$ is a solution of \eqref{eq:SEKI_flow_linear}. Since $x_{\tau_n} \to \hat{x}$ strongly in $C([0,T];\mathbb{R}^{d\cdot J})$ and $\big(\frac{dx_{\tau_n}}{dt}\big)$ is bounded in $L^2([0,T];\mathbb{R}^{d\cdot J})$, the limit trajectory $\hat{x}$ is absolutely continuous on $[0,T]$.
    \medskip
    
    Uniqueness follows from \Cref{prop:uniqueness}, which completes the proof.
\end{proof}

\section{From continous- to discrete-time: Subgradient ensemble Kalman inversion}\label{sec:discrete}
The continuous-time formulation \eqref{eq:SEKI_flow_linear} is best understood as an idealized model that reveals the underlying covariance-preconditioned (sub)gradient-flow structure and enables a clean well-posedness analysis via maximal monotone operator theory. However, it does not uniquely prescribe a practical time discretization. In particular, the preconditioner $C(t)$ is time-dependent and could also be low-rank, so discretizations based on implicit steps or operator inversions, such as proximal or splitting schemes, are not straightforward to implement in this setting. Proximal or splitting schemes would require evaluating the proximal operator in the time-dependent metric induced by $C^{-1}(t)$ which involves inversion of the empirical covariance.

For this reason, the continuous-time analysis is restricted to structural properties and well-posedness, while convergence is studied at the level of a discrete-time scheme that is directly implementable. This distinction is further motivated by the fact that certain properties available in continuous time, most notably constraint viability for indicator-type regularizers, may fail under naive explicit discretizations. The discrete-time method therefore serves as the primary algorithmic object, whereas the continuous-time dynamics provides a principled guideline for its design.

Motivated by these considerations, we focus on an explicit forward discretization, which preserves the derivative-free structure and avoids matrix inversions involving the empirical covariance. The resulting subgradient-based ensemble Kalman update reads
\begin{equation} \label{eq:SEKI_discrete}
x_{k+1}^{(j)} = x_k^{(j)} - h\, C_k A^\top \Gamma^{-1} (A x_k^{(j)} - y)
- h\, C_k g_{\bar x_k},
\end{equation}
where $g_{\bar x_k} \in \partial R(\bar x_k)$ denotes a measurable selection of the subdifferential, $h>0$ is the step size, and 
\[
C_k = C(x_k) = \frac{1}{J}\sum_{j=1}^J (x_k^{(j)} - \bar x_k)(x_k^{(j)} - \bar x_k)^\top
\]
is the empirical covariance matrix.

\begin{assumption}\label{ass:smoothloss}
There exist constants $\mu, L > 0$ with $\mu \le L$ such that
\[
\mu I \preceq S := A^\top \Gamma^{-1} A \preceq L I,
\]
meaning that the data misfit $\Phi$ is $\mu$-strongly convex and $L$-smooth.
\end{assumption}

\subsection{Ensemble Collapse}
The ensemble collapse is studied in the following section. The analysis follows closely to Section~3 in \cite{weissmann2024ensemble}. We revisit several technical steps to be self-contained and make certain implicit assumptions explicit, resulting in slightly more restrictive bounds on $h$. \medskip

Define the ensemble spread of the particles $e_k^{(j)}:=x_k^{(j)}-\bar x_k$, $j=1,\dots,J$, $k\in\N$ and consider its evolution
\begin{equation}\label{eq:spread}
    e_{k+1}^{(j)} = x_{k}^j-\bar x_k - h C_k A^\top \Gamma^{-1}(Ax_k^{(j)}-A\bar x_k) = e_k^{(j)} - h C_k A^\top \Gamma^{-1} Ae_k^{(j)}\,.
\end{equation}
The aim is to quantify upper and lower bounds on the eigenvalues of $C_k$. Note that the sample covariance evolves as
\begin{equation}\label{eq:samplecov}
    C_{k+1} = C_k - 2h C_kA^\top AC_k + h^2 C_k A^\top A C_k A^\top A C_k = C_k - 2hC_k S C_k + h^2 C_k S C_k S C_k\,.
\end{equation}
It is useful to define $E_k := \frac{1}{J} \sum_{j=1}^J \|e_k^{(j)}\|^2$ which is related to the Frobenius norm of $C_k$ as follows
\begin{align*} 
\frac1JE_k^2 = \frac1J\Big(\frac1J\sum_{j=1}^J\|e_k^{(j)}\|^2\Big)^2\le \frac1{J^2} \sum_{j=1}^J\|e_k^{(j)}\|^4 \le  \frac{1}{J^2}\sum_{j,l=1}^J \langle e_k^{(j)},e_k^{(l)}\rangle^2  &= \|C_k\|_{\cF}^2\\ &\le \frac{1}{J^2}\sum_{j,l=1}^J\|e_k^{(j)}\|^2 \|e_k^{(l)}\|^2 = E_k^2\,, 
\end{align*}
where we used Jensen's inequality in the first inequality and Cauchy-Schwarz inequality in the last inequality.\medskip

The next lemma describes lower and upper bounds on the sample covariance.
\begin{lemma}\label{lem:cov_discretetime}
    Suppose that \Cref{ass:regularization} and \Cref{ass:smoothloss} are in place. Let $(x_k^{(j)},\, k\in\mathbb{N})_{j=1,\dots,J}$ be generated by \eqref{eq:SEKI_discrete} with fixed step size $h>0$ and 
initial ensemble $(x_0^{(j)})_{j=1,\dots,J}$, such that $C_0\succ \sigma_0 I$ for some $\sigma_l>0$. Moreover, let 
$$
h\le \min\Bigg(\frac{\mu}{JL^2},\frac{J}{\mu},\frac{1}{2L}\Bigg) E_0^{-1}.
$$ 
Then for all $k\ge0$ it holds true that \[ \frac{\sigma_l}{k+1} I \preceq C_k \preceq \min\Big(\frac{\sigma_u}{(k+1)},E_0\Big) I = \min\Big(\frac{J}{h\mu(k+1)},E_0\Big) I\,,\]
for $\sigma_u = \frac{J}{h\mu}$ and $\sigma_l = \frac{\sigma_0}{1+2hL\sigma_0}$. 
\end{lemma}
\begin{proof}
    We begin with the upper bound and prove the lower bound in the second step.\medskip
    
    \textit{\underline{Step 1:} Upper bound.} The iteration of $E_k$ can be written and bounded as follows:
    \begin{align*}
        E_{k+1} = E_k - 2h \frac{1}{J}\sum_{j=1}^J \langle e_k^{(j)},C_k S e_k^{(j)}\rangle + h^2 \|C_k Se_j^{(j)}\|^2 &\le E_k - 2h \mu \|C_k\|_{\cF}^2 + h^2 L^2 \|C_k\|_{\cF}^2 E_k\\
        &\le E_k - \frac{2h\mu}J E_k^2 + h^2L^2 E_k^3\,.
    \end{align*}
    Provided that $h\le \frac{\mu}{JL^2} E_0^{-1}$ it holds $h^2L^2E_0^3\le \frac{h\mu}{J}E_0^2$ and hence, $E_1\le E_0$. This implies $h\le \frac{\mu}{JL^2} E_1^{-1}$ and therefore, that $E_2\le E_1$. By induction, we have 
    \[h^2 L^2 E_k^3 \le \frac{h\mu}{J} E_k^2\]
    for all $k\in\N$ meaning that $E_k$ is decreasing with
    \[ E_{k+1} \le E_k-\frac{h\mu}{J}E_k^2\,.\]
    If $E_{k_0} = 0$ for some $k_0\in\N$, then $E_k = 0$ for all $k\ge k_0$ and the to be derived upper bounds always hold. Thus, without loss of generality assume that $E_k > 0$ for all $k\in\N$, then 
    \[ \frac{1}{E_{k+1}}\ge \frac{1}{E_k} \frac{1}{1-\frac{h\mu}{J} E_k}\,.\]
    By assumption $\frac{h\mu}{J} E_k\le \frac{h\mu}{J} E_0 < 1$ holds so that, using $(1-u)^{-1} \ge 1+u$ for $u\in[0,1)$, we deduce that
    \[ \frac{1}{E_{k+1}}\ge \frac{1}{E_k} + h \ge \dots \ge \frac{1}{E_0} + h(k+1)\,.\]
    This verifies that 
    \[E_k \le \frac{1}{\frac{h\mu}{J} k + \frac{1}{E_0}}\,.\]
    Note that $\frac{h\mu}{J} k + \frac{1}{E_0} \ge (\frac{1}{E_0}\wedge \frac{h\mu}{J}) (k+1) \ge \frac{h\mu}{J} (k+1)$ by the assumption on $h$. In summary, we have verified that
    \begin{equation}\label{eq:upperboundensemblespread} 
    E_k \le \frac{\sigma_u}{k+1},\quad \sigma_u = \frac{J}{h\mu}
    \end{equation}
    and therefore, 
    \[ C_{k} \preceq \min\big(\frac{\sigma_u}{k+1},E_0\Big) I\,.\]\medskip
    
    \textit{\underline{Step 2:} Lower bound.} For the lower bound on $C_k$, assume that $C_0\succeq \sigma_0 I$ for some $\sigma_0>0$. The goal is to select $h>0$ sufficiently small such that
    \begin{equation}\label{eq:suffcond}
    C_k S C_k \succeq C_k S C_k S C_k\,.
    \end{equation}
    In order to do so, define $Q_k = C_k^{\frac12}S C_k^{\frac12}$ so that \eqref{eq:suffcond} writes as
    \[C_k^{\frac12} Q_k C_k^{\frac12} \succeq h C_k^{\frac12}Q_k^2 C_k^{\frac12}\,.\]
    Under the assumption that $C_k$ is invertible, we require $I-hQ_k\succeq 0$ that produces the sufficient condition $h\le \frac{1}{\lambda_{\max}(Q_k)}$, where $\lambda_{\max}(Q_k)$ denotes the largest eigenvalue of $Q_k$. Since $\lambda_{\max}(C_k)\le E_k\le E_0$ and $\lambda_{\max}(S)\le L$, our assumption $h\le \frac{1}{LE_0}$ is sufficient to guarantee \eqref{eq:suffcond} as long as $C_k$ is invertible. Hence, provided that $\lambda_{\min}(C_k)>0$ it holds
    \begin{equation} \label{eq:lowerbound_cov1}
    C_{k+1} \succeq C_k - h C_kSC_k \succeq C_k - h L C_k^2\,.
    \end{equation}
    Diagonalize $C_k = U_k D_k U_k^\top$ with $D_k = {\mathrm{diag}}_{i=1,\dots,d}(\lambda_k^{(i)})$ where $U_k$ are orthogonal matrices and $\lambda_k^{(i)}\ge \lambda_{\min}(C_k)>0$, $i=1,\dots,d$ are the eigenvalues of $C_k$. Now \eqref{eq:lowerbound_cov1} reads as
    \begin{equation} \label{eq:lowerbound_cov2}
    C_{k+1} \succeq U_k{\mathrm{diag}}_{i=1,\dots,d}\big(\lambda_k^{(i)}-hL(\lambda_k^{(i)})^2 \big)U_k^\top\,.
    \end{equation}
    By assumption $h\le \frac{1}{2LE_0}$ such that $\lambda_{\max}(C_k)\le \lambda_{\max}(C_0) \le \frac{1}{2Lh}$ and, hence, 
    $\lambda_k^{(i)}\in[0,\frac{1}{2hL}]$ for all $i=1,\dots,d$. Since $g(\lambda) := \lambda - hL\lambda^2$ is decreasing on $[0,\frac{1}{2Lh}]$, we have 
    \[\min_{i=1,\dots,d} \lambda_k^{(i)}-hL(\lambda_k^{(i)})^2 = \lambda_{\min}(C_k) - hL \lambda_{\min}^2(C_k)\,.\]
    This shows
    \[ q_{k+1}:=\lambda_{\min}(C_{k+1})\ge \min_{i=1,\dots,d} \lambda_k^{(i)}-hL(\lambda_k^{(i)})^2 \ge \lambda_{\min}(C_k)-hL\lambda_{\min}^2(C_k) = q_k-hLq_k^2\,.\]
    By assumption $h\le \frac{1}{2LE_0}$, such that $hLq_k\le hL E_0\le \frac{1}{2}$ for all $k\in\N$ implying that
    \[ q_{k+1} \ge \frac{1}{2}q_k \ge \frac{1}{2^{k+1}} q_0 \ge \frac{1}{2^{k+1}}\sigma_0 >0\,.\]
    We can now deduce that
    \[ \frac{1}{q_{k+1}}\le \frac{1}{q_k}\frac1{1-hLq_k} \le \frac{1}{q_k}(1+2hLq_k) \le \frac{1}{q_k} + 2hL \le \dots \le \frac{1}{q_0} + 2hL(k+1)\]
    where we used $(1-u)^{-1}\le 1+2u$ for all $u\in(0,1/2]$ with $hLq_k\le 1/2$. This verifies the lower bound
    \[\lambda_{\min}(C_k)\ge \frac{1}{2hL k + \sigma_0^{-1}} \ge \frac{1}{(\sigma_0^{-1} + 2hL)(k+1)} = \frac{\sigma_l}{k+1}\]
    with $\sigma_l := (\sigma_0^{-1} + 2hL)^{-1} = \frac{\sigma_0}{1+2hL\sigma_0}$ and concludes the proof.
\end{proof}

\begin{lemma}\label{lem:cov_discretetime_cauchy}
    Under the setting of \Cref{lem:cov_discretetime}, it holds true that
\[\|C_{k+1}-C_{k}\|_{\cF}\le \Big(\frac{2LJ^2}{h\mu^2}+\frac{L^2 J^3}{h\mu^3}\Big)\frac{1}{(t+1)^2}\,. \]
\end{lemma}
\begin{proof}
    Using the evolution \eqref{eq:samplecov} we find that
    \begin{align*}
        \|C_{k+1}-C_k\|_{\cF} \le 2hL E_k^2 + h^2 L^2 E_k^3 \le \Big(\frac{2LJ^2}{h\mu^2}+\frac{L^2 J^3}{h\mu^3}\Big)\frac{1}{(t+1)^2}
    \end{align*}
    where we used the upper bound on $E_k$ derived in \eqref{eq:upperboundensemblespread}.
\end{proof}

\subsection{Convergence analysis}
We are now ready to formulate the main result of this section. While well-posedness has been established in the continuous-time setting, the convergence towards the minimizer of of $\Phi_R$ is verified in the discrete-time setting. The proof follows the general outline of convergence proofs for sub-gradient descent in the strongly convex regime. However, preconditioning through $C_k$ involves multiple challenges that can be solved through covariance control from \Cref{lem:cov_discretetime,lem:cov_discretetime_cauchy}. 
\begin{theorem}\label{thm:main_convergence_discr}
    Suppose that \Cref{ass:regularization} and \Cref{ass:smoothloss} are in place and assume that subgradients of $R$ are uniformly bounded, i.e., there exists $M>0$ such that
    \[
    \|g_x\| \le M
    \quad \text{for all } g_x \in \partial R(x), \; x\in\R^d.
    \]
    Let $(x_k^{(j)},\, k\in\mathbb{N})_{j=1,\dots,J}$ be generated by \eqref{eq:SEKI_discrete} with
    initial ensemble $(x_0^{(j)})_{j=1,\dots,J}$, such that $C_0\succ \sigma_0 I$ for some $\sigma_l>0$, and
    let $h\le \frac{\mu}{4L^2\lambda_{\max}(C_0)}$. Then there exists $k_0\ge0$ such that for $K\ge k_0$
    \[ \frac{1}{K-k_0} \sum_{k=k_0}^K \Big(2h(\Phi_R(\bar x_k)-\Phi_R(x_\ast)) + \frac{\mu^2}{16L^2}\|\bar x_k-x_\ast\|^2\Big) \in \cO\Big(\frac{\log((K+1)/(k_0+1))}{K-k_0}\Big)\,,\]
    where $x_\ast = \arg\min_{x\in\R^d} \, \Phi_R(x)$.
\end{theorem}

\begin{proof}
    Recall that $\Phi(x):=\frac12\|Ax-y\|^2,\ x\in\R^d$ and consider
    \begin{align*} 
    \bar x_{k+1} = \bar x_k - h C_k A^\top \Gamma^{-1} (A\bar x_k-y) - hC_k g_{\bar x_k} = \bar x_k - h C_k \nabla \Phi(x_k) - h C_k g_{\bar x_k}\,.
    \end{align*}
    Since $\Phi$ is strongly convex, there exists a unique global minimizer $x_\ast = \arg\min_{x\in\R^d} \, \Phi_R(x)$. Consider the evolution of the error $r_k^{(1)} = \|\bar x_k - x_\ast\|$ which satisfies
    \[ r_{k+1}^{(1)} = \|\bar x_{k+1} - x_\ast\|^2 = \|x_k-x_\ast\|^2 - 2h \langle C_k (\nabla \Phi(\bar x_k)+ g_{\bar x_k}), \bar x_k-x_\ast\rangle + h^2 \|C_k(\nabla \Phi(\bar x_k) +g_{\bar x_k})\|^2\,.\]
    Note that $C_k$ is invertible for all $k\ge0$ by \Cref{lem:cov_discretetime}, so that we can compute
    \begin{align*}
        \|\bar x_{k+1} - x_\ast\|_{C_k}^2 &= \|\bar x_k-x_\ast\|_{C_k}^2 - 2h\langle C^{-1}(x_k) C_k (\nabla \Phi(\bar x_k)+g_{\bar x_k}), \bar x_k-x_\ast\rangle\\
        &\quad + h^2 \|C^{-1/2}(x_k) C_k (\nabla \Phi(\bar x_k)+g_{\bar x_k})\|^2\\
        &= \|\bar x_k-x_\ast\|_{C_k}^2 - 2 h\langle \nabla \Phi(\bar x_k)+g_{\bar x_k},\bar x_k-x_\ast\rangle + h^2 \|C^{1/2}(x_k) (\nabla \Phi(\bar x_k)+g_{\bar x_k})\|^2\\
        &\le \|\bar x_k-x_\ast\|_{C_k}^2 - 2h \big(\Phi(\bar x_k)+R(\bar x_k) -(\Phi(x_\ast)+R(x_\ast)\big) - h\mu\|\bar x_k-x_\ast\|^2\\
        & \quad + h^2 \lambda_{\max}(C(x_k)) \|\nabla \Phi(\bar x_k)+ g_{\bar x_k}- \nabla \Phi(x_\ast)- g_{x_\ast}\|^2\,,
    \end{align*}
    where we used 
    \[\Phi(z)+R(z) \ge \Phi(x) + R(x) + \langle z-x,\nabla \Phi(x)+g_x\rangle + \frac{\mu}2 \|x-z\|^2\] 
    for all $x,z\in\R^d$ and all $g_x\in\partial R(x)$ due to convexity. Moreover, we used the optimality condition $0 \in \nabla \Phi(x_\ast)- \partial R(x_\ast)$. Rewrite the last line and compute
    \begin{align*}
        \|\bar x_{k+1} - x_\ast\|_{C_k}^2 &\le \|\bar x_k-x_\ast\|_{C_k}^2 - 2h (\Phi_R(\bar x_k)-\Phi_R(x_\ast)) - h\mu \|\bar x_k-x_\ast\|^2\\
        &\quad+  2h^2\lambda_{\max}(C_k) \|\nabla \Phi(\bar x_k)-\nabla \Phi(\bar x_\ast)\|^2 + 2h^2 \lambda_{\max}(C_k) \|g_{\bar x_k}-g_{\bar x_\ast}\|^2\\
        &\le\|\bar x_k-x_\ast\|_{C_k}^2 - 2h (\Phi_R(\bar x_k)-\Phi_R(x_\ast))+ (2h^2 \lambda_{\max}(C_k) L^2 - h \mu)\|\bar x_k-x_\ast\|^2\\ &\quad + 2h^2 \lambda_{\max}(C_k) M^2\,,
    \end{align*}
    where we used the $L$-Lipschitz continuity of $\nabla \Phi$ and the boundedness of the subgradients $g_x$ for all $g_x\in\partial R(x)$ and all $x\in\R^d$. Since
    \begin{align*} 
    \|\bar x_k-x_\ast \|_{C_k}^2  &= \|\bar x_k-x_\ast\|_{C_{k-1}} + \langle \bar x_k-x_\ast , (C^{-1}_k-C^{-1}_{k-1})(\bar x_k-x_\ast)\rangle\\
    &\le \|\bar x_k-x_\ast\|_{C_{k-1}} + \|C^{-1}_k-C^{-1}_{k-1}\| \|\bar x_k-x_\ast\|^2
    \end{align*}
    it holds 
     \begin{align*}
        \|\bar x_{k+1} - x_\ast\|_{C_k}^2 &\le \|\bar x_k-x_\ast\|_{C_{k-1}}^2 - 2h (f(\bar x_k)-f(x_\ast))+ 2h^2 \lambda_{\max}(C_k) M^2\\ &\quad+ (2h^2 \lambda_{\max}(C_k) L^2 + \|C^{-1}_k-C^{-1}_{k-1}\| - h \mu)\|\bar x_k-x_\ast\|^2\,. 
    \end{align*}

    In the remainder of the proof, assume that the step size satisfies $h\le \frac{\mu}{4L^2\lambda_{\max}(C_0)}$ such that $2h(-\mu + 2 h \lambda_{\max}(C_k))\le -\frac{\mu^2}{8L^2}$. By \Cref{lem:cov_discretetime_cauchy}, there exists $k_0\ge 0$ such that $\|C^{-1}_k-C^{-1}_{k-1}\| \le \frac{\mu^2}{16L^2}$ for all $k\ge k_0$ implying 
    \[ \|\bar x_{k+1} - x_\ast\|_{C_k}^2 \le \|\bar x_k-x_\ast\|_{C_{k-1}}^2- \frac{\mu^2}{16L^2}\|\bar x_k-x_\ast\|^2 - 2h (\Phi_R(\bar x_k)-\Phi_R(x_\ast))+ 2h^2 \lambda_{\max}(C_k) M^2 \,.\]
    for all $k\ge k_0$. Rearranging this inequality and taking a summation from $k=k_0,\dots, K$ yields
    \begin{align*}
        \frac{1}{K-k_0} &\sum_{k=k_0}^K \Big(2h(\Phi_R(\bar x_k)-\Phi_R(x_\ast)) + \frac{\mu^2}{16L^2}\|\bar x_k-x_\ast\|^2\Big) \\ 
        &\le \frac{\sum_{k=k_0}^K \|\bar x_{k}-x_\ast\|_{C(x_{k-1})}^2 - \|\bar x_{k+1}-x_\ast \|_{C_k}^2}{K-k_0} + \frac{2h^2 M^2\sum_{k=k_0}^K \lambda_{\max}(C_k)}{K-k_0} \\
        &\le \frac{\|x_{k_0}-x_\ast\|_{C(x_{k_0-1})}^2}{K-k_0} +  {\frac{2hM^2J}{\mu}}\frac{2+\log((K+1)/(k_0+1))}{K-k_0}\,, 
    \end{align*}
    where we have used 
    \[ \sum_{k=k_0}^K \lambda_{\max}(C_k)\le \frac{J}{h\mu} \sum_{k=k_0}^K\frac{1}{k+1}  \le {\frac{J}{h\mu}} \Big(\frac{1}{k_0+1}+\int_{k_0}^{K}\frac{1}{{x+1}}\,\dd x \Big)\le {\frac{J}{h\mu}}  \Big(2+\log\big((K+1)/(k_0+1)\big)\Big)\,\]
    by \Cref{lem:cov_discretetime}. 
\end{proof}

\section{Numerical Experiments}\label{sec:numerics}
The goal of the numerical experiments is not to advocate EKI as a superior optimizer compared with specialized proximal or primal--dual methods. Instead, we aim to show that non-smooth regularization can be incorporated into ensemble Kalman inversion in a stable and principled way through the proposed subgradient framework. We compare the proposed method primarily with a baseline subgradient descent method in order to isolate the effect of the covariance-based preconditioning.

We begin by introducing a practical hybrid implementation of \eqref{eq:SEKI_discrete} in which the empirical covariance is frozen after a burn-in phase. We then study the method on two representative examples: computed tomography with total variation regularization and compressed sensing with $\ell_1$-regularization.

\subsection{Hybrid SEKI with covariance freezing}
To better understand the role of covariance-based preconditioning and to enable a fair comparison of computational costs, we consider a hybrid implementation of SEKI. The strategy consists of an initial \emph{burn-in phase} (Phase I) during which the full SEKI dynamics are evolved to learn a problem-adapted preconditioner from the ensemble. Subsequently, in the \emph{freezing phase} (Phase II), the empirical covariance is fixed, and only the ensemble mean is updated using the resulting static preconditioner.

This hybrid approach effectively transitions from a particle-based ensemble method to a preconditioned subgradient iteration. Crucially, the scheme remains adjoint-free throughout both phases, as all required quantities are derived solely from forward model evaluations and ensemble averages. By freezing the covariance, the per-iteration cost in Phase II is reduced from $J$ forward evaluations to a single evaluation at the mean iterate. The hybrid scheme therefore provides a convenient way to isolate the effect of the covariance-based preconditioning learned by SEKI while substantially reducing the computational cost allowing for the use of large ensemble sizes. 
We summarize this strategy in \Cref{alg:frozen_seki}. 

Note that Phase II requires the specification of a sequence of step sizes $(h_k)_{k \in \mathbb{N}_0}$. We typically employ a diminishing step-size rule, e.g., $h_k \sim (k+1)^{-p}$ for $p \in (0, 1]$, to reflect the decaying update magnitudes observed during the intrinsic ensemble collapse of standard SEKI. For the case $p=1$, the convergence results under strong convexity (\Cref{thm:main_convergence_discr}) extend naturally to this hybrid setting.

\begin{algorithm}[htb!]
\caption{Hybrid SEKI with covariance freezing}
\label{alg:frozen_seki}
\begin{algorithmic}[1]

\Require Ensemble $\{x^{(j)}_0\}_{j=1}^J\subset \R^d$, step sizes $\{h_k\}_{k\ge0}$, burn-in length $k_{\mathrm b}\in\N$, total length $K\ge k_{\mathrm b}$, data $y$, (augmented) forward operator $A$, covariance $\Gamma$, regularizer $R$
\Ensure Hybrid iterates $\bar x_k$

\State \textbf{Phase I (SEKI burn-in phase):}
\For{$k=0,\dots,k_{\mathrm b}-1$}

\State compute mean 
$
\bar x_k=\frac1J\sum_{j=1}^J x^{(j)}_k
$;
\State compute covariances
$
C_k=\frac1J\sum_{j=1}^J (x^{(j)}_k-\bar x_k)(x^{(j)}_k-\bar x_k)^\top\,,$\\ \hspace{4cm} $
C_k^{x,A}=\frac1J\sum_{j=1}^J (x^{(j)}_k-\bar x_k)(Ax^{(j)}_k-A\bar x_k)^\top\, ;
$
\State pick $\bar g_k\in\partial R(\bar x_k)$;

\State update particles \quad 
$
x^{(j)}_{k+1}
= x^{(j)}_{k}
- h_k\, C_k^{x,A}\, \Gamma^{-1}(A x^{(j)}_{k}-y)
- h_k\, C_k\, \bar g_k
$,\quad $j=1,\dots,J$;

\EndFor

\Statex

\State \textbf{Freeze preconditioner:}
store $\widehat C := C_{k_{\mathrm b}}$ and $\widehat C^{x,A} := C_{k_{\mathrm b}}^{x,A}
$;

\Statex

\State \textbf{Phase II (covariance freezing phase):}

\State initialize
$
\bar x_{k_{\mathrm b}}=\frac1J\sum_{j=1}^J x^{(j)}_{k_{\mathrm b}}
$\,

\For{$k=k_{\mathrm b},\dots,K-1$}

\State pick $\bar g_k\in\partial R(\bar x_k)$;

\State update mean
$
\bar x_{k+1}
= \bar x_k
- h_k\, \widehat C^{x,A}\, \Gamma^{-1}(A\bar x_k-y)
- h_k\, \widehat C\, \bar g_k
$.

\EndFor

\end{algorithmic}
\end{algorithm}

\subsection{Image reconstruction based on the radon transform}
\label{sec:num_radon}
We consider a computed tomography inverse problem based on the Radon transform \cite{Hansen2021}. 
The unknown $x^\dagger\in\R^d$ represents a two-dimensional image obtained by vectorizing an $n\times n$ phantom ($d=n^2$). 
The forward operator $A\in\R^{K\times d}$ is a discrete Radon transform that maps an image to a collection of line integrals corresponding to prescribed projection angles and detector bins.

Given $x^\dagger$, we generate noisy observations according to
\[
y = A x^\dagger + \eta, \qquad \eta\sim \mathcal N(0,\sigma^2 I),\quad \sigma=0.01.
\]

To reconstruct the image, we solve the variational problem
\begin{equation}\label{eq:Opt_Radon}
\min_{x\in\R^d}\ \Phi_R(x),\quad \Phi_R(x):= \Phi(x)+R_{\mathrm{Tik}}(x)+R_{\mathrm{TV}}(x).
\end{equation}
Here
\[
\Phi(x):=\frac{1}{2\sigma^2}\|Ax-y\|^2,\qquad 
R_{\mathrm{Tik}}(x)=\frac{0.01}{2}\|x\|^2,\qquad 
R_{\mathrm{TV}}(x)=0.1\,\mathrm{TV}(x),
\]
where the total variation regularizer is defined as
\[
\mathrm{TV}(x):=\sum_{i,j}\bigl(|x_{i+1,j}-x_{i,j}|+|x_{i,j+1}-x_{i,j}|\bigr)
\]
for $x\in\R^d$ representing a discretized image of size $\sqrt{d}\times\sqrt{d}$ pixels.

\begin{figure}[!htb]
  \centering \includegraphics[width=1.0\textwidth]{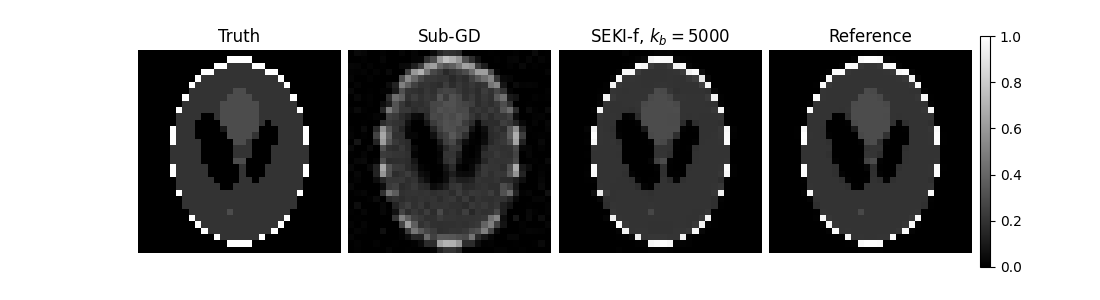}
 \caption{Reconstruction of the unknown image $x^\dagger$ using Sub-GD and SEKI-f ($k_b=5000$) after $5\cdot10^5$ iterations. From left to right: ground truth, Sub-GD reconstruction, SEKI-f reconstruction (ensemble mean), and the reference solution $x_\ast$.} \label{fig:radon_reconcomp}
\end{figure} 

\paragraph{Implementation details.} 
The experiment was implemented in \texttt{Python} using the \textit{scikit-image} library for the Radon transform. 
Our simulations compare the hybrid SEKI scheme of Algorithm~\ref{alg:frozen_seki} (SEKI-f) with step sizes
\begin{equation}\label{eq:SEKI_stepradon}
h_k =
\begin{cases}
h_0, & k=0,\dots,k_b-1,\\
\displaystyle \frac{\lambda_{\max}(C_{k_b-1})\,k_b\,h_0}{(k+1)^p}, & k=k_b,\dots,K,
\end{cases}
\end{equation}
and a baseline subgradient descent (Sub-GD) method that uses the same step-size schedule $h_k = \frac{h_0}{(k+1)^p}$ for all $k=0,\dots,K-1$:
\[
x_{k+1}=x_k - h_k\Big(A^\top \Gamma^{-1}(Ax_k-y) + g_k\Big),\qquad g_k\in\partial R(x_k).
\]
The factor $\lambda_{\max}(C_{k_b-1})k_b$ serves as an optimistic estimate of the constant $\sigma_u$ appearing in the upper bound on the sample covariance $C_k$ in Lemma~\ref{lem:cov_discretetime}. 
Since all methods use the same initial step size $h_0$ and a comparable decay rate for $h_k$, the observed differences can primarily be attributed to covariance-based preconditioning.

In our implementation we set $h_0 = 0.9\cdot \lambda_{\max}(A^\top A + 0.01\,I_d)$ and $p=1$, which is justified by the strong convexity of the objective.
SEKI-f is initialized with i.i.d.~particles $x_0^{(j)}\sim \mathcal N(0,0.1^2 I_d)$, while Sub-GD is initialized with their empirical mean
\[
\bar x_0 = \frac{1}{J}\sum_{j=1}^J x_0^{(j)}.
\]

To compare the computational cost of ensemble-based and mean-based methods, we report performance as a function of both the \emph{iteration count} and the \emph{wall-clock time}.
One SEKI iteration during the burn-in phase requires $J$ evaluations of the forward map $v\mapsto Av$, whereas SEKI-f after the burn-in phase and Sub-GD require only a single evaluation per iteration.

As performance metrics we track the relative reconstruction error $\|x_k-x_\ast\|/\|x_\ast\|$ (using the ensemble mean $\bar x_k$ for SEKI-f) and the optimality gap $\Phi_R(x_k)-\Phi_R(x_\ast)$. 
Here, $x^\ast$ denotes an approximate minimizer of \eqref{eq:Opt_Radon} obtained from a substantially longer run of Sub-GD. 
Further implementation details are summarized in \Cref{tab:radon_params} of \Cref{app:num_radon}.

\begin{figure}[!htb]
  \centering \includegraphics[width=0.45\textwidth]{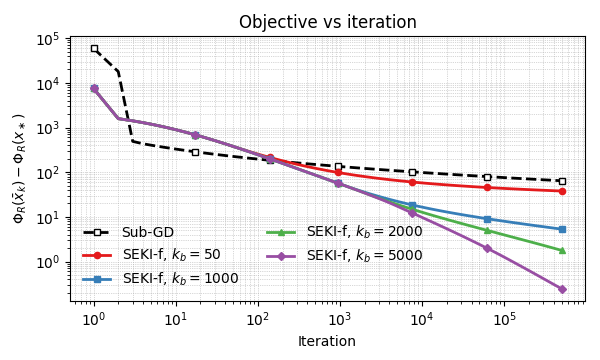}~~\includegraphics[width=0.45\textwidth]{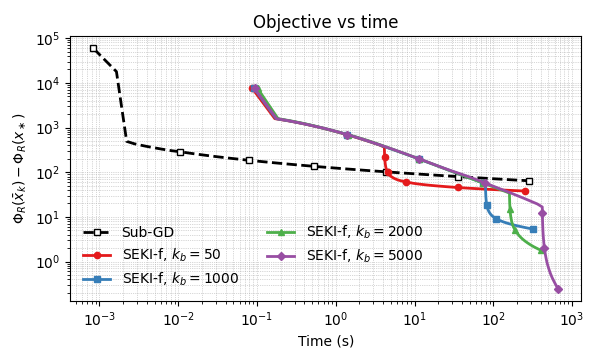}
 \caption{Objective gap $\Phi_R(x_k)-\Phi_R(x_\ast)$ for Sub-GD and SEKI-f with burn-in lengths $k_b\in\{50,1000,2000,5000\}$. Left: convergence versus iteration. Right: convergence versus computational time.} \label{fig:radon_loss}
\end{figure} 

\begin{figure}[!htb]
  \centering \includegraphics[width=0.45\textwidth]{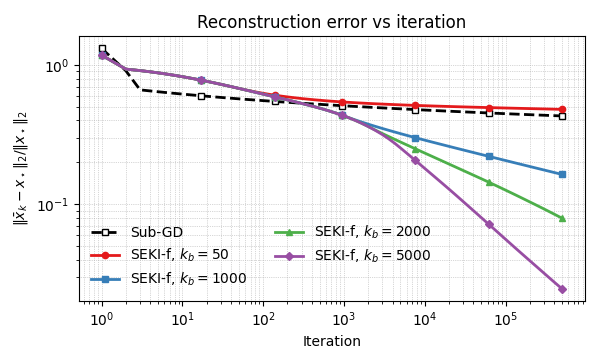}~~\includegraphics[width=0.45\textwidth]{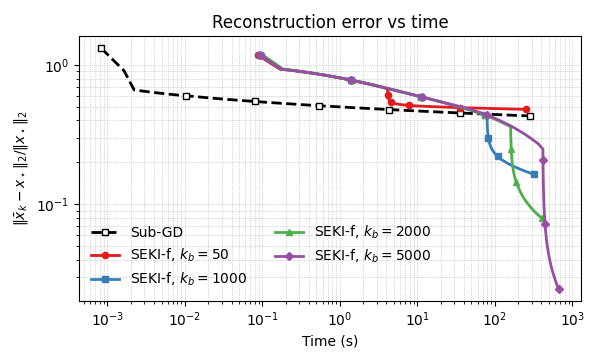}
 \caption{Relative reconstruction error $\|\bar x_k-x_\ast\|/\|x_\ast\|$ for Sub-GD and SEKI-f with burn-in lengths $k_b\in\{50,1000,2000,5000\}$. Left: error versus iteration. Right: error versus computational time.} \label{fig:radon_recon}
\end{figure}

\paragraph{Results.} 
The results show that both the objective gap (\Cref{fig:radon_loss}) and the reconstruction error (\Cref{fig:radon_recon}) decrease over time. 
In terms of iteration complexity, SEKI-f exhibits faster convergence than Sub-GD. 
This behavior is also visible in the right panels of the figures, where performance is plotted against computational time.
In \Cref{fig:radon_reconcomp} we additionally display the reconstructions obtained after $5\cdot10^5$ iterations. Since Sub-GD and SEKI-f use similar step-size schedules, the improved performance of SEKI-f can be attributed to the covariance-based preconditioning.
This effect becomes more pronounced when comparing different burn-in lengths $k_b\in\{50,1000,2000,5000\}$. 
While larger values of $k_b$ increase the computational cost of the burn-in phase, they lead to substantially better performance in the frozen phase, as illustrated in \Cref{fig:radon_loss} and \Cref{fig:radon_recon}.

\subsection{Compressed sensing with correlated sensing matrices}
\label{sec:num_cs}
As a second example, we consider the recovery of a sparse signal in a linear compressed sensing model. 
This setup allows us to investigate the effect of $\ell_1$-regularization for sparse recovery.

Consider the linear observation model
\begin{equation}\label{eq:IP_cs}
y = A x^\dagger + \eta,
\qquad
\eta \sim \mathcal N(0,\sigma^2 I),\quad \sigma=0.02,
\end{equation}
where $A\in\mathbb R^{K\times d}$ is a sensing matrix and $x^\dagger\in\mathbb R^d$ is a sparse signal. 
To introduce correlations and control the conditioning of the problem, we generate $A$ as
\[
A = B \Sigma^{1/2},
\]
where $B\in\mathbb R^{K\times d}$ has entries drawn independently from 
$\mathcal N(0,1/K)$ and $\Sigma\in\mathbb R^{d\times d}$ is a correlation matrix with entries
\[
\Sigma_{ij} = \rho^{|i-j|}, \qquad i,j=1,\dots,d.
\]
The parameter $\rho\in[0,1)$ controls the correlation between the columns of $A$. 
For $\rho=0$, the sensing matrix reduces to the standard Gaussian design, while larger values of $\rho$ lead to increasingly correlated columns and a more ill-conditioned forward operator. This construction allows us to systematically vary the conditioning of the inverse problem through the parameter $\rho$.

The sparse ground truth $x^\dagger$ is generated by selecting a support set 
$S\subset\{1,\dots,d\}$ of size $s$ uniformly at random and setting
\[
(x^\dagger)_i =
\begin{cases}
a_i, & i\in S,\\
0, & i\notin S,
\end{cases}
\]
where the amplitudes $a_i$ are drawn independently from $\mathcal N(0,1)$. 
The observations $y$ are then generated according to the model \eqref{eq:IP_cs}.

To recover the sparse signal, we solve the $\ell_1$-regularized least-squares problem
\[
\min_{x\in\mathbb R^d}
\; \Phi_R(x),\qquad 
\Phi_R(x):=
\Phi(x) + R(x)
=
\frac12\|Ax-y\|^2 + \alpha \|x\|_1,
\]
with $\alpha=0.1$. 
The subgradients of the $\ell_1$-norm are chosen component-wise as $g_x\in\partial\|x\|_1$ with
\[
(g_x)_i =
\begin{cases}
\operatorname{sign}(x_i), & x_i\neq 0,\\
0, & x_i=0 .
\end{cases}
\] 

\begin{figure}[!htb]
  \centering \includegraphics[width=0.85\textwidth]{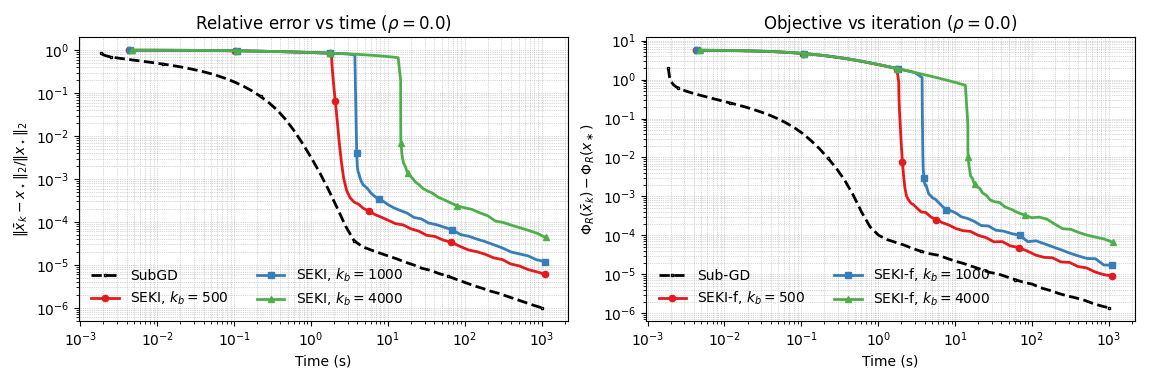}
 \caption{Relative reconstruction error (left) and objective gap (right) for the compressed sensing experiment with correlation parameter $\rho=0$. Results are shown as a function of computational time for Sub-GD and SEKI-f with burn-in lengths $k_b\in\{500,2000,4000\}$. } \label{fig:rho1}
\end{figure}

\paragraph{Implementation details.}
Our simulations again compare the proposed hybrid SEKI method from Algorithm~\ref{alg:frozen_seki} (SEKI-f) with subgradient descent (Sub-GD). 
SEKI-f is initialized with i.i.d.~particles $x_0^{(j)}\sim \mathcal N(0,0.1^2\,I_d)$, while Sub-GD is initialized with their empirical mean $\bar x_0 = \frac{1}{J}\sum_{j=1}^J x_0^{(j)}$. The step sizes for SEKI-f are again defined via \eqref{eq:SEKI_stepradon}, while Sub-GD uses a diminishing schedule of the form $h_k = h_0/(k+1)^p$ for all $k=0,\dots,K-1$. 
In both cases we set $h_0 = 0.9\cdot \lambda_{\max}(A^\top A)$ and $p = 0.6$, since the objective $\Phi_R$ is not strongly convex.

To assess the influence of ill-conditioning, we repeat the experiment for several values of the correlation parameter $\rho\in\{0,0.95,0.98\}$ while keeping the ground truth $x^\dagger$ fixed. 
For each method we report the reconstruction error relative to a high-accuracy reference solution $x_\ast$, obtained by running a proximal gradient method (ISTA) for a large number of iterations. 
In particular, we measure $\|x_k - x_{\ast}\|/\|x_{\ast}\|$ as well as the objective gap $\Phi_R(x_k)-\Phi_R(x_{\ast})$. 
Further implementation details are summarized in \Cref{tab:cs_params} of \Cref{app:num_cs}.

\begin{figure}[!htb]
  \centering \includegraphics[width=0.85\textwidth]{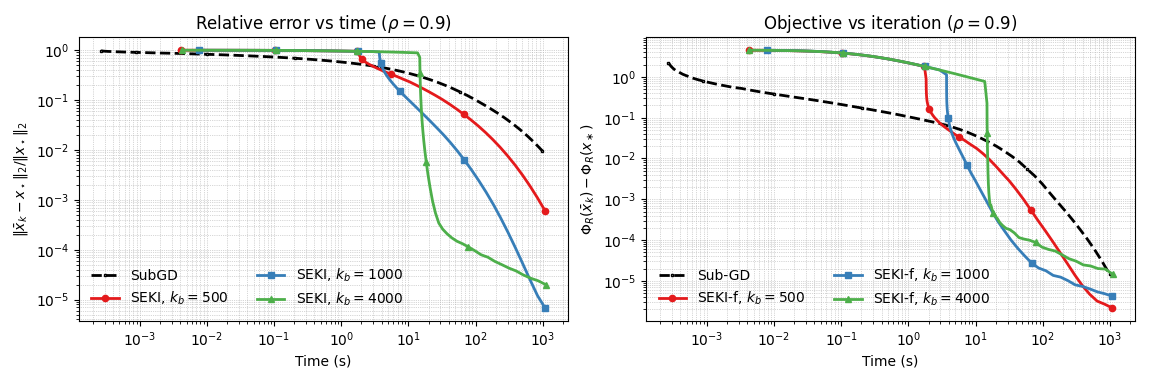}
 \caption{Relative reconstruction error (left) and objective gap (right) for the compressed sensing experiment with correlation parameter $\rho=0.95$. Results are shown as a function of computational time for Sub-GD and SEKI-f with burn-in lengths $k_b\in\{500,2000,4000\}$. } \label{fig:rho2}
\end{figure} 

\begin{figure}[!htb]
\centering \includegraphics[width=0.85\textwidth]{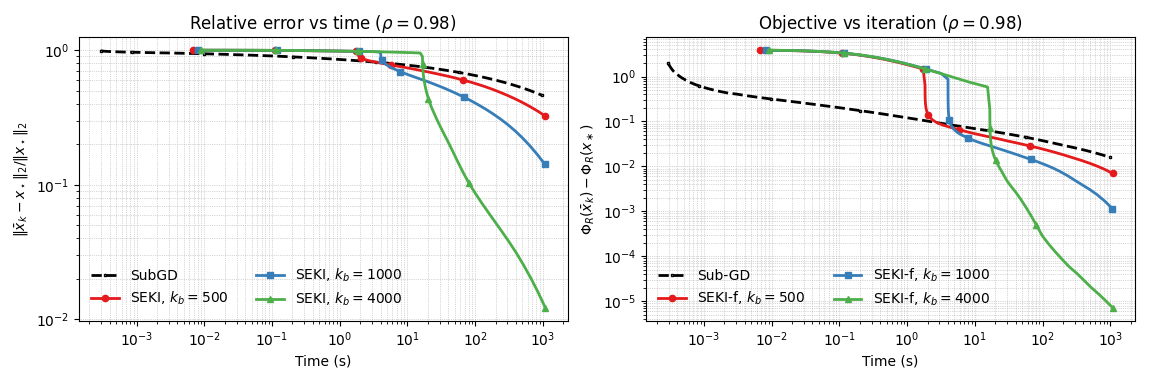}
 \caption{Relative reconstruction error (left) and objective gap (right) for the compressed sensing experiment with correlation parameter $\rho=0.98$. Results are shown as a function of computational time for Sub-GD and SEKI-f with burn-in lengths $k_b\in\{500,1000,4000\}$. } \label{fig:rho3}
\end{figure} 

\paragraph{Results.}
We consider three different choices of the correlation parameter $\rho$. 
Larger values of $\rho$ lead to a more ill-conditioned optimization problem. 
For $\rho=0$, Sub-GD significantly outperforms SEKI-f (\Cref{fig:rho1}). 
However, as the conditioning deteriorates, the covariance-based preconditioning in SEKI-f becomes increasingly advantageous, which is reflected in the results for $\rho\in\{0.95,0.98\}$ (\Cref{fig:rho2,fig:rho3}).

In these figures, we report the relative reconstruction error and the objective gap as functions of the computational time. We additionally compare different burn-in lengths $k_b\in\{500,2000,4000\}$ to illustrate the trade-off between computational cost and the quality of the learned preconditioner.

Finally, \Cref{fig:rho1_recon,fig:rho2_recon,fig:rho3_recon} of \Cref{app:num_cs} compare the true signal with the corresponding reconstructions for $\rho\in\{0,0.95,0.98\}$. 
In the right panels, the signal is plotted after sorting the entries according to the values of $x^\dagger$. As expected, the reconstruction quality deteriorates for larger values of $\rho$, which reflects the increasing difficulty of the inverse problem under the fixed regularization parameter.

\section{Conclusion}\label{sec:conclusion}
In this work, we introduced a subgradient-based formulation of EKI for variational inverse problems with convex, possibly non-smooth regularization. The proposed SEKI dynamics provide a covariance-preconditioned subgradient flow that extends the classical gradient-flow interpretation of deterministic EKI. In the linear setting we established well-posedness of the resulting particle system and proved convergence of a practical explicit discretization of the method. Numerical experiments demonstrate that non-smooth regularizers such as total variation and $\ell_1$ penalties can be incorporated into EKI within this framework.

Several directions for future work remain open. A first important step is the extension of the theoretical analysis to nonlinear forward models. Furthermore, it would be interesting to investigate discretizations based on splitting or proximal-type schemes that more directly exploit the structure of the non-smooth regularization. Finally, a deeper understanding of the covariance freezing strategy would be desirable, in particular adaptive approaches that detect when the covariance structure has stabilized while its overall magnitude continues to decay.

\bibliographystyle{plain}
\bibliography{references, references_old}

\appendix

\section{Additional information}

\subsection{Background: Solutions of differential inclusions}

We begin with the notion of a solution to a differential inclusion. This definition follows the standard notion of solutions used in \cite{aubin1984differential}. 
\begin{definition}[Solution of a differential inclusion]
Let $F : [0,T] \times \mathbb R^d \rightrightarrows \mathbb R^d$
be a set-valued map with nonempty closed values and measurable in $t$.
A function $x : [0,T] \to \mathbb R^d$ is called a \emph{solution}
of the differential inclusion
\begin{equation}\label{eq:DI}
    \frac{{\mathrm{d}} x(t)}{{\mathrm d}t} \in F(t,x(t))
\end{equation}
if the following conditions hold:
\begin{enumerate}
\item $x$ is absolutely continuous on $[0,T]$;
\item there exists a measurable function
      $w : [0,T] \to \mathbb R^d$ such that
      \begin{equation}\label{eq:selection}
          w(t) \in F(t,x(t))
          \quad \text{for almost every } t \in [0,T];
      \end{equation}
\item for almost every $t \in [0,T]$ the differential equality holds:
      \begin{equation}\label{eq:DI-ODE}
          \frac{{\mathrm{d}} x(t)}{{\mathrm d}t} = w(t)\,.
      \end{equation}
\end{enumerate}
The function $w(\cdot)$ is called a \emph{measurable selection}
of the right-hand side along the trajectory $x(\cdot)$.
\end{definition}

\subsection{Implementation details of \Cref{sec:num_radon}}\label{app:num_radon}
The following \Cref{tab:radon_params} provides all implementation details of the Radon transform experiment.

\begin{table}[!htb]
\centering
\caption{Experimental parameters for the Radon transform tomography experiment.}
\label{tab:radon_params}
\begin{tabular}{ll}
\hline
Parameter & Value \\ 
\hline
Image size & $32 \times 32$ ($d = 1024$ unknowns) \\
Forward operator & Discrete Radon transform $A$ \\
Number of projection angles & $50$ \\
Detector bins per angle & $32$ \\
Data dimension & $K = 1600$ \\
Noise model & $y = Ax^\dagger + \eta$, $\eta\sim\mathcal N(0,\sigma^2 I)$ \\
Noise level & $\sigma = 0.01$ \\
Regularizers & Tikhonov: $R_{\mathrm{Tik}}(x)=\frac{\alpha_1}{2}\|x\|^2$, $\alpha_1 = 0.01$ \\
             & Total Variation: $R_{\mathrm{TV}}(x)=\alpha_2\,\mathrm{TV}(x)$, $\alpha_2 = 0.1$ \\
Ensemble size & $J = 1200$ \\
SEKI-f step size (Phase I) & $h_k = h_0=0.9\cdot \lambda_{\max}(A^\top A + 0.01\cdot I_d)$ \\
SEKI-f step size (Phase II) & $h_k = k_b\, h_0/(k+1)$ \\
Burn-in iterations (freezing) & $k_b \in \{50,1000,2000,5000\}$ \\
Initialization & $x_0^{(j)} \sim \mathcal N(0, 0.1^2\cdot I_d)$ \\
\hline
\end{tabular}
\end{table}

\subsection{Implementation details and additional results of \Cref{sec:num_cs}}\label{app:num_cs}
The following \Cref{tab:cs_params} provides all implementation details of the compressed sensing experiment. Moreover, in \Cref{fig:rho1_recon,fig:rho2_recon,fig:rho3_recon} below the resulting reconstructions of the sparse signal are displayed for different values of $\rho$.

\begin{table}[!htb]
\centering
\caption{Experimental parameters for the compressed sensing benchmark.}
\label{tab:cs_params}
\begin{tabular}{ll}
\hline
Parameter & Value \\ 
\hline
Signal dimension & $d = 512$ \\
Number of measurements & $K = 160$ \\
Forward operator & $A = G \Sigma^{1/2}$ with $G_{ij}\sim \mathcal N(0,1/K)$ \\
Column correlation & $\Sigma_{ij} = \rho^{|i-j|}$ \\
Correlation parameter & $\rho \in \{0,0.9,0.98\}$ \\
Sparsity level & $s = 20$ nonzero coefficients in $x^\dagger$ \\
Signal amplitudes & $(x^\dagger)_i \sim \mathcal N(0,1)$ for $i\in S$ \\
Noise model & $y = Ax^\dagger + \eta$, $\eta\sim\mathcal N(0,\sigma^2 I)$ \\
Noise level & $\sigma = 0.02$ \\
Regularizer & $R(x)=\alpha\|x\|_1$, $\alpha = 0.05$ \\
Ensemble size & $J = 1500$ \\
SEKI-f step size (Phase I) & $h_k = h_0 = 0.9\cdot\lambda_{\max}(A^\top A)$ \\
SEKI-f step size (Phase II) & $h_k = k_b\,h_0 / (k+1)^{0.6}$\\
Burn-in iterations (freezing) & $k_b \in \{500,1000,4000\}$ \\
Initialization & $x_0^{(j)} \sim \mathcal N(0,0.1^2 I)$ for $j=1,\dots,J$ \\
\hline
\end{tabular}
\end{table}

\begin{figure}[!htb]
  \centering \includegraphics[width=0.8\textwidth]{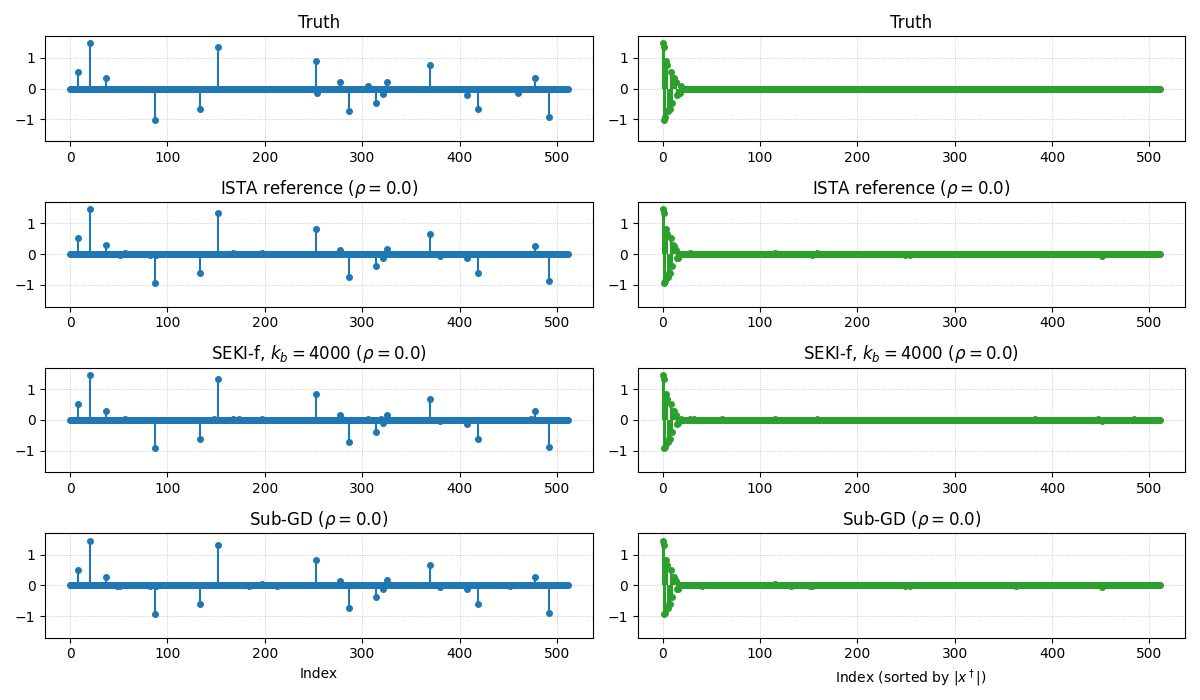}
 \caption{Recovery of the sparse signal for $\rho=0$. From top to bottom: ground truth, ISTA reference solution, SEKI-f ($k_b=4000$), and Sub-GD. Left: signal in original order. Right: entries sorted by the magnitude of $x^\dagger$. } \label{fig:rho1_recon}
\end{figure} 

\begin{figure}[!htb]
  \centering \includegraphics[width=0.8\textwidth]{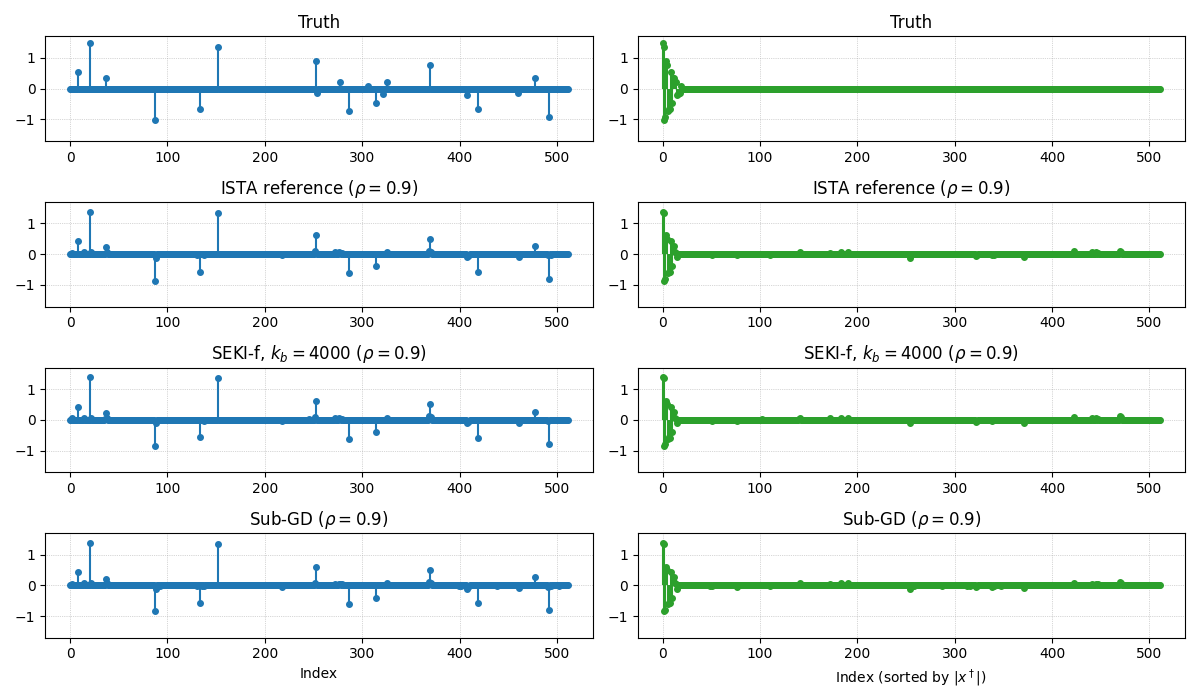}
 \caption{Recovery of the sparse signal for $\rho=0.95$. From top to bottom: ground truth, ISTA reference solution, SEKI-f ($k_b=4000$), and Sub-GD. Left: signal in original order. Right: entries sorted by the magnitude of $x^\dagger$. } \label{fig:rho2_recon}
\end{figure} 

\begin{figure}[!htb]
  \centering \includegraphics[width=0.8\textwidth]{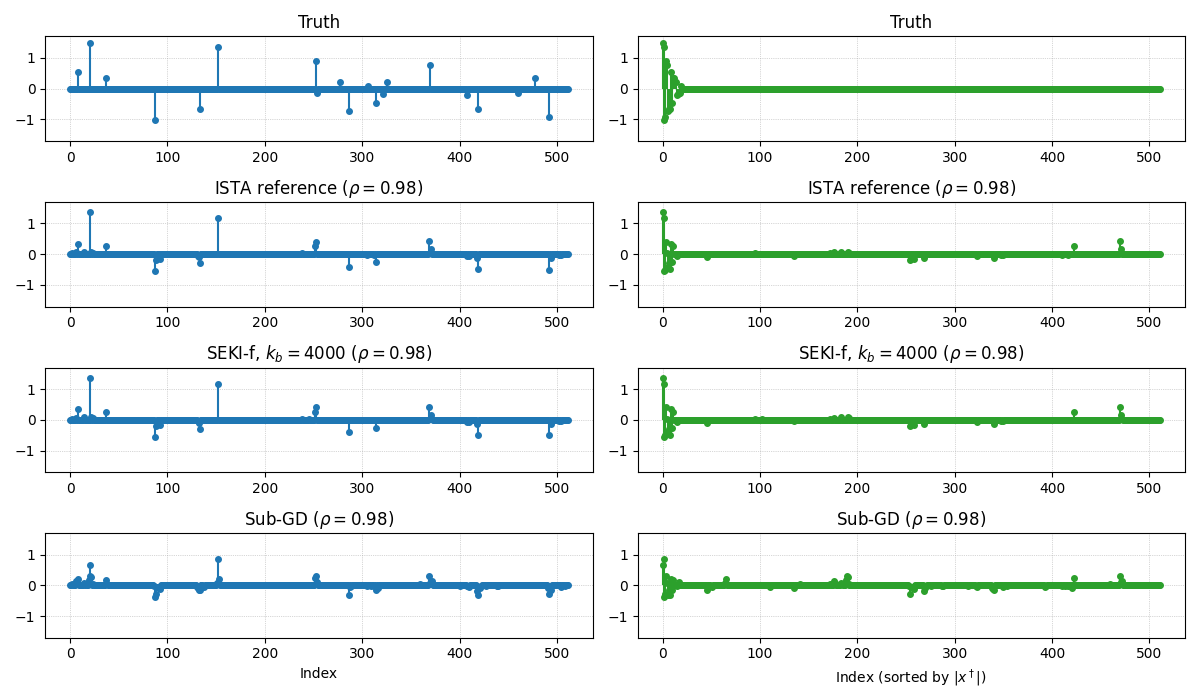}
 \caption{Recovery of the sparse signal for $\rho=0.98$. From top to bottom: ground truth, ISTA reference solution, SEKI-f ($k_b=4000$), and Sub-GD. Left: signal in original order. Right: entries sorted by the magnitude of $x^\dagger$.} \label{fig:rho3_recon}
\end{figure}

\end{document}